\documentclass[12pt, reqno]{amsart}

\usepackage{mathrsfs}
\usepackage{amssymb, amsmath, amsthm}
\usepackage[backref]{hyperref}
\usepackage[alphabetic,backrefs,lite]{amsrefs}
\usepackage{amscd}   
\usepackage{fullpage}
\usepackage[all]{xy} 

\DeclareFontEncoding{OT2}{}{} 

\usepackage[usenames,dvipsnames]{color}

\newtheorem{lemma}{Lemma}[section]
\newtheorem{theorem}[lemma]{Theorem}
\newtheorem{proposition}[lemma]{Proposition}
\newtheorem{prop}[lemma]{Proposition}
\newtheorem{cor}[lemma]{Corollary}

\newtheorem{claim*}{Claim}
\newtheorem{thm}[lemma]{Theorem}

\newtheorem{question}[lemma]{Question}

\theoremstyle{definition}
\newtheorem{remark}[lemma]{Remark}
\newtheorem{remarks}[lemma]{Remarks}

\newtheorem{example}[lemma]{Example}

\newcommand{\A}{{\mathbb A}}
\newcommand{\G}{{\mathbb G}}

\newcommand{\PP}{{\mathbb P}}

\newcommand{\F}{{\mathbb F}}
\newcommand{\Q}{{\mathbb Q}}

\newcommand{\Z}{{\mathbb Z}}

\newcommand{\Xbar}{{\overline{X}}}

\newcommand{\kbar}{{\overline{k}}}
\newcommand{\Kbar}{{\overline{K}}}

\newcommand{\Ubar}{{\overline{U}}}

\newcommand{\Adeles}{{\mathbb A}}
\newcommand{\kk}{{\mathbf k}}

\newcommand{\eps}{{\varepsilon}}

\newcommand{\scrS}{{\mathscr S}}
\newcommand{\scrT}{{\mathscr T}}

\newcommand{\calA}{{\mathcal A}}
\newcommand{\calB}{{\mathcal B}}

\newcommand{\calD}{{\mathcal D}}

\newcommand{\calO}{{\mathcal O}}

\newcommand{\calX}{{\mathcal X}}

\DeclareMathOperator{\HH}{H}

\DeclareMathOperator{\rk}{rk}

\DeclareMathOperator{\inv}{inv}

\DeclareMathOperator{\im}{im}

\DeclareMathOperator{\Aut}{Aut}
\DeclareMathOperator{\Gal}{Gal}

\DeclareMathOperator{\Norm}{N}
\DeclareMathOperator{\Br}{Br}

\DeclareMathOperator{\divv}{div}
\DeclareMathOperator{\ord}{ord}

\DeclareMathOperator{\Div}{Div}
\DeclareMathOperator{\Pic}{Pic}
\DeclareMathOperator{\Jac}{Jac}

\DeclareMathOperator{\ev}{ev}

\DeclareMathOperator{\et}{et}

\DeclareMathOperator{\id}{id}

\DeclareMathOperator{\verti}{vert}

\newcommand{\isom}{\cong}

\numberwithin{equation}{section}
\numberwithin{table}{section}

\newcommand{\defi}[1]{\textsf{#1}} 

\title{Vertical Brauer groups and del Pezzo surfaces of degree $4$}
\author{Anthony V\'arilly-Alvarado}
\author{Bianca Viray}
\thanks{The first author was partially supported by NSF grant DMS-1103659, the Centre Interfacultaire Bernoulli, and the \'Ecole Polytechnique F\'ed\'eral de Lausanne.  The second author was partially supported by NSF grant DMS-1002933 and the Centre Interfacultaire Bernoulli. }

\address{Department of Mathematics MS 136, Rice University, Houston, TX
			77005, USA}
\address{Section de math\'ematiques EPFL, FSB-SMA, Station 8 - B\^atiment MA, CH-1015 Lausanne}
\email{varilly@rice.edu}
\urladdr{http://math.rice.edu/\~{}av15}

\address{Department of Mathematics, Box 1917, Brown University, Providence, RI
			02912, USA}
\email{bviray@math.brown.edu}
\urladdr{http://math.brown.edu/\~{}bviray}

\date{}

\begin{document}

	\begin{abstract}
		We show that Brauer classes of a locally solvable degree $4$ del Pezzo surface $X$ are \defi{vertical} for some projection away from a plane $g\colon X \dasharrow \PP^1$, i.e., that every Brauer class is obtained by pullback from an element of $\Br \kk(\PP^1)$. As a consequence, we prove that a Brauer class  obstructs the existence of a $k$-rational point if and only if all $k$-fibers of $g$ fail to be locally solvable, or in other words, if and only if $X$ is covered by curves that each have no adelic points. The proof is constructive and gives a simple and practical algorithm, distinct from that in~\cite{BBFL}, for computing all  classes in the Brauer group of $X$ (modulo constant algebras).
	\end{abstract}

	\maketitle
	
	\vspace{-.18in}
	
	\section{Introduction}\label{sec:Intro}

		Let $X$ be a smooth projective geometrically integral variety over a global field $k$; write $X(\A_k)$ for its adelic points, and let $\kk(X)$ denote its function field. In 1970, Manin used the Brauer group $\Br(X) := \HH^2_{\et}(X,\G_m)$ to define the \defi{Brauer set} $X(\A_k)^{\Br} \subseteq X(\A_k)$. He proved that this set contains the $k$-rational points of $X$, and that it can thus obstruct the existence of rational points~\cite{Manin-ICM}. This obstruction is known as the \defi{Brauer-Manin obstruction}.

		 In some examples, the Brauer group is ``vertical'', and it is possible to interpret the Brauer-Manin obstruction through rational maps. For instance, consider the degree $4$ del Pezzo surface $X$ over $\Q$ given as the complete intersection of the following two quadrics in $\PP^4$:
		 \begin{equation}
		 	\label{BSD}
			\begin{split}
		 	x_3x_4 &= x_2^2 - 5x_0^2, \\
		 	(x_3 + x_4)(x_3 + 2x_4) &= x_2^2 - 5x_1^2. 
			\end{split}
		 \end{equation}
		 Birch and Swinnerton-Dyer showed that $X(\A_\Q)^{\Br} = \emptyset$~\cite[Thm.\ 3]{BSD-dP4}. Furthermore, in this example, $\Br X \subseteq f^*\left(\Br\kk(\PP^1)\right)$, where $f\colon X \dasharrow \PP^1$ is the map $[x_0 : \cdots : x_4] \mapsto [x_3:x_4]$. Together, these facts imply that $f^{-1}(t)(\A_{\Q}) = \emptyset$ for all $t \in \PP^1(\Q)$, and thus the map $f$ witnesses that each adelic point of $X$ is not arranged in a way that is globally compatible.

		It is natural to ask for which classes of varieties $X$ one is guaranteed to have a \defi{vertical Brauer group}, i.e., that $\Br X \subseteq f^*\left(\Br\kk(Y)\right)$, for some dominant map $f\colon X\dasharrow Y$. We show that locally solvable del Pezzo surfaces of degree $4$ have this property.

	\begin{thm}\label{thm:MainIntro}
		Let $k$ be a global field of characteristic not $2$, and let $X$ be a del Pezzo surface of degree $4$ over $k$ such that $X(\A_k)\neq\emptyset$.  
		Then for all $\calA\in\Br X$, there exists a map $g = g_{\calA}\colon X\dasharrow \PP^1$, {obtained by projecting from a plane,} with at most two geometrically reducible fibers such that $\calA\in g^*\left(\Br\kk(\PP^1)\right)$. Moreover, there exists a map $f\colon X\dasharrow \PP^n$, {obtained by projecting from a linear space,} such that $\Br X$ is vertical with respect to $f$, i.e., 
		\[
			\Br X \subseteq f^*\left(\Br\kk(\PP^n)\right),
		\]
		where $n = \dim_{\F_2} \left(\Br X/\im \Br k\right)[2]$.
	\end{thm}

		\begin{remarks}\ 
			\begin{enumerate}
				\item For a del Pezzo surface of degree $4$, the group $\Br X/\im \Br k$ is isomorphic to either $0$, $\Z/2\Z$ or $\Z/2\Z\times \Z/2\Z$~\cite{Manin-CubicForms,SD-BrauerGroup}.
				\item Theorem~\ref{thm:MainIntro} holds over all fields of characteristic not $2$, after replacing the hypothesis that $X(\A_k) \neq \emptyset$ with an assumption on the rank 4 quadrics in a pencil associated to $X$; see~\S\ref{sec:VerticalElements} for the precise condition. 
				\item The proof of Theorem~\ref{thm:MainIntro} is constructive, and  therefore gives an algorithm for computing explicit Brauer classes of $X$. In contrast to the existing algorithm developed by Bright, Bruin, Flynn, and Logan~\cite{BBFL}, our method does not proceed by computing the Galois action on the set of exceptional curves of $X$. Our algorithm is easily explained and quite fast in practice; it essentially only requires finding a rational point (which is known to exist) on a few rank 4 quadrics. See~\S\ref{subsec:Algorithm} for details.
			\end{enumerate}
		\end{remarks}

		\noindent {\bf Related work.}
			\begin{enumerate}
				\item Colliot-Th\'el\`ene, Harari, and Skorobogatov have studied a similar question: they proved a criterion~\cite[Prop.~2.5]{CTHS-NormicBundle} for determining when $\Br X \subseteq f^*\left(\Br\kk(\PP^1)\right)$ for $f\colon X \to \PP^1$ a normic bundle, and gave several concrete examples of bundles satisfying their criterion~\cite[Cor.~2.6]{CTHS-NormicBundle}. The salient difference between this article and theirs is that they work with an {already specified} morphism $f$.
				\item Any element in $\Br X$ of the form $(L/k,f)$, where $L/k$ is a cyclic extension and $f \in k(X)^\times$, is plainly vertical for the map $X \dasharrow \PP^1,\ x\mapsto f(x)$. Swinnerton-Dyer showed that all nonconstant Brauer classes on a del Pezzo surface of degree $4$ are of this form~\cite[Ex. 3]{SD99}. However, his functions $f \in k(X)^\times$ are ratios of polynomials whose degrees are not bounded by the construction. Moreover, the computation of these functions, and hence the computation of the map $X\dasharrow \PP^1$, generally requires an explicit Galois descent of divisor classes.
			\end{enumerate}

		\smallskip

		As in the example above, the Brauer-Manin obstruction to the existence of $k$-points on $X$  manifests itself through the fibers of a map. 

		\begin{cor}\label{cor:FiberManifestation}
			Retain the notation from Theorem~\ref{thm:MainIntro}. Assume that $X(\A_k) \neq\emptyset$ and $X(\A_k)^{\Br} = \emptyset$.  Then 
			\[
				f^{-1}(t)(\A_k) =\emptyset \quad \forall t \in \PP^n(k).
			\]
		\end{cor}

		\begin{remark}
			A stronger statement is true. If $n > 0$, then the map $f\colon X \dasharrow \PP^n$ in Theorem~\ref{thm:MainIntro} is given by projecting away from a codimension $n+1$ linear subspace $W$. The proof of Corollary~\ref{cor:FiberManifestation} shows that $(X \cap L)(\A_k) = \emptyset$ for any codimension $n$ linear space $L$ that contains $W$. Since $f^{-1}(t) \subseteq X\cap L$ for some such $L$, the corollary follows from this stronger statement.
		\end{remark}

		In addition, using work of Colliot-Th\'el\`ene and Skorobogatov~\cite{CTSko-FibrationRevisited} and their earlier work with Swinnerton-Dyer~\cite{CTSkoSD-Crelle} (which both build on~\cite{CTSansuc-Schinzel,CTSD-HPWAForSB}), we also obtain a partial converse to Corollary~\ref{cor:FiberManifestation}.

		\begin{cor}\label{cor:FiberEquivalence}
			Retain the notation from Theorem~\ref{thm:MainIntro}. Assume that $X(\A_k) \neq\emptyset$ and $\Br X/\im \Br k = \langle\calA\rangle\isom \Z/2\Z$.  Then
			\[
				\overline{\{t\in\PP^1(k) : {g_\calA}^{-1}(t)(\A_k)\neq\emptyset\}} = 
				{g_\calA}(X(\A_k)^{\Br}),
			\]
			where the closure takes place in the adelic topology.
		\end{cor}

		A special case of an important conjecture in the study of rational points on varieties, due to Colliot-Th\'el\`ene and Sansuc~\cite{CTS80}, says that for a degree $4$ del Pezzo surface $X$, we have $X(k) \neq \emptyset$ as soon as $X(\A_k)^{\Br} \neq \emptyset$. Assuming Schinzel's hypothesis and the finiteness of Tate-Shafarevich groups for elliptic curves, Wittenberg proved the conjecture for a general such $X$ over a number field $k$, as well as in many other cases~\cite[Thm.~3.36]{Wittenberg-SpringerThesis}. All the cases handled by Wittenberg require that $\Br X = \Br k$, and the proof proceeds by showing that certain fibrations of genus $1$ curves satisfy the Hasse principle. 

		If $\Br X/\im\Br k \cong \Z/2\Z$, then the generic fiber of the map $f$ given in the proof of Theorem~\ref{thm:MainIntro} is a curve of genus $1$. We show that for certain del Pezzo surfaces over number fields Wittenberg's method, which builds on ideas of Swinnerton-Dyer~\cite{SD-TwoQuadrics,BenderSD} that were formalized in~\cite{CTSkoSD-Inventiones,CT-BenderSD}, applies to these fibrations, and therefore $\overline{X(k)} = X(\A_k)^{\Br}$ for these surfaces, conditional on Schinzel's hypothesis and finiteness of Tate-Shafarevich groups. Interestingly, to apply Wittenberg's method, we must fix $f$ such that there are \emph{exactly} $6$ geometrically reducible fibers. This is in contrast to the condition that $g_\calA$ has at most $2$ geometrically reducible fibers, which is needed to prove Corollary~\ref{cor:FiberEquivalence}.  Indeed, this requirement of additional geometrically reducible fibers is part of the reason why Wittenberg's method, {in its current form}, does not apply to all degree $4$ del Pezzo surfaces with $\Br X/\im\Br k \cong \Z/2\Z$.

		\subsection*{Outline} We fix notation in \S\ref{subsec:Notation}. In \S\ref{sec:PicGens} we study the rank $4$ quadrics in the pencil associated to a del Pezzo surface $X$ of degree $4$, and use~\cite{BBFL} to give a convenient set of generators for the geometric Picard group of $X$. In \S\ref{sec:PicGalois}, we use~\cite{KST-dP4s} to describe the Galois action on these generators. 
		
		The heart of the paper is \S\ref{sec:VerticalElements}. In it, we construct explicit Brauer classes (\S\ref{subsec:Construction}) and use them to prove an explicit version of Theorem~\ref{thm:MainIntro} for all fields of characteristic not $2$ (\S\S\ref{subsec:Nontriviality},\ref{subsec:MainResult}). In \S\ref{sec:Computations}, we discuss computational applications of the results from \S\ref{sec:VerticalElements}: in \S\ref{subsec:Algorithm}, we describe an algorithm to compute nonconstant Brauer classes of $X$, and in \S\ref{subsec:ParameterSpaces} we show how to construct parameter spaces of degree $4$ del Pezzo surfaces with a nonconstant Brauer class.
		
		In \S\ref{sec:Arithmetic} we restrict our attention to the case where the base field is a global field of characteristic not $2$, and prove Theorem~\ref{thm:MainIntro} and Corollaries~\ref{cor:FiberManifestation} and~\ref{cor:FiberEquivalence}. In addition, in \S\ref{subsec:Eval}, we prove that the evaluation maps $\ev_\calA$ for $\calA \in \Br X$ are constant for a large set of places of $k$. In \S\ref{subsec:ConditionD} we show that the Brauer-Manin obstruction to the Hasse principle and weak approximation is the only one for certain del Pezzo surfaces of degree $4$ over number fields, assuming Schinzel's hypothesis and finiteness of Tate-Shafarevich groups. Finally, in \S\ref{sec:Order4}, we discuss a partial improvement to Theorem~\ref{thm:MainIntro} in the case where $\#(\Br X/\im\Br k) = 4$.
		
		\section*{Acknowledgements} We thank David Harari for pointing out that many of the arguments in this paper hold for an arbitrary field of characteristic not $2$, and not just for global fields of characteristic not 2. We also thank Olivier Wittenberg for pointing out Remark~\ref{rmk:VerticalConic}, as well as for discussions outlining how to use~\cite{Wittenberg-SpringerThesis} to prove Theorem~\ref{thm:BM only one}, and Jean-Louis Colliot-Th\'el\`ene for helpful comments.  Much of this paper was written during the semester program ``Rational points and algebraic cycles'' at the Centre Interfacultaire Bernoulli; we are grateful to the staff and the organizers for their support.
		
	\section{The Picard group of del Pezzo surfaces of degree $4$}
	\label{sec:Picdp4}

		\subsection{Notation}\label{subsec:Notation}

			We fix much of the notation that will remain in force throughout the paper. Let $k$ be a field of characteristic not $2$, and fix a separable closure $\kbar$ of $k$. Let $G_k$ denote the absolute Galois group $\Gal(\kbar/k)$. For any homogeneous polynomial $F$ in five variables, we write $V(F)$ for the corresponding variety in $\PP^4$. For any smooth projective geometrically integral $k$-variety, we use $\sim$ to denote linear equivalence of divisors over $\kbar$.

			Let $X$ be a del Pezzo surface of degree $4$ over a field $k$ of characteristic not $2$.  By embedding it anticanonically, we will view $X$ as a smooth complete intersection of two quadrics, $Q$ and $\widetilde{Q}$, in $\PP^4$.  The forms $Q$ and $\widetilde{Q}$ define a pencil $\{\lambda Q + \mu \widetilde{Q} : [\lambda:\mu] \in \PP^1\}$ with five degenerate geometric fibers, each a rank $4$ quadric~\cite[Prop. 3.26]{Wittenberg-SpringerThesis}.  

			We write $\scrS\subseteq\PP^1$ for the degree $5$ subscheme defining the degeneracy locus of the pencil. We denote the $\kbar$-points of $\scrS$ by $t_0, \ldots, t_4$; let $k(t_i)$ be the smallest subfield of $\kbar$ that contains $t_i$, and let $Q_i$ denote the rank $4$ quadric associated to $t_i$, which is defined over $k(t_i)$.  We write $\eps_{t_i}$ for the discriminant of the smooth rank 4 quadric obtained by restricting $V(Q_i)$ to a hyperplane $H \subseteq \PP^4$ not containing the vertex of $V(Q_i)$. Throughout, we consider $\eps_{t_i}$ as an element of $k(t_i)^{\times}/k(t_i)^{\times2}$; as such, it does not depend on the choice of $H$~\cite[\S3.4.1]{Wittenberg-SpringerThesis}.

			Similarly, for any $T \in \scrS$, we write $\kappa(T)$ for the residue field of $T$, $Q_T$ for the corresponding rank 4 quadric, and $\eps_T \in \kappa(T)^{\times}/\kappa(T)^{\times2}$ for the discriminant of the restriction of $V(Q_T)$ to a suitable hyperplane as above.

			As in the introduction, the Brauer group of $X$ is $\Br X := \HH^2_{\et}(X,\G_m)$. We write $\Br_0 X := \im \Br k \to \Br X$ for the subgroup of constant algebras. For a dominant map $f\colon X \dasharrow Y$, we define the vertical Brauer group of $X$ with respect to $f$ as
			\[
				\Br^{(f)}_{\verti} X := \Br X \cap f^*\left(\Br \kk(Y)\right).
			\]
			More generally, a map $f\colon X\dasharrow Y$ gives rise to pullback map $f^*\colon \Br \calO_{Y,f(\eta)} \to \Br \kk(X)$, where $\calO_{Y,f(\eta)}$ is the local ring of $Y$ at the image of the generic point $\eta$ of $X$. In this case, we define
			\[
				\Br^{(f)}_{\verti} X = \Br X \cap f^*\left(\Br\calO_{Y,f(\eta)}\right).
			\]

		\subsection{Linear spaces on rank $4$ quadrics}\label{sec:Rank4}

		\begin{lemma}\label{lem:Rank4NormalForm}
			Assume that $V(Q_T)$ has a smooth $\kappa(T)$-point $P$.  Then there exists a set of linearly independent $\kappa(T)$-linear forms $\ell_1, \ell_2, \ell_3,$ and $\ell_4$ such that $\ell_2(P) = \ell_3(P) = \ell_4(P) = 0$ and
			\[
				cQ_T = \ell_1\ell_2 - \ell_3^2 + \eps_T\ell_4^2,
			\]
			for some $c \in \kappa(T)^{\times}$. In particular, if $H$ is a hyperplane tangent to $V(Q_T)$ at $P$, then $X\cap H$ is geometrically reducible and $V(Q_T)\cap H = L \cup L'$ for some planes $L, L'$ defined over $\kappa(T)\left(\sqrt{\eps_T}\right)$.
		\end{lemma}

		\begin{proof}
		 	There are exactly two planes $L$, $L'$ through $P$ that lie on $V(Q_T)$. Since $P$ is defined over $\kappa(T)$, the absolute Galois group $G_{\kappa(T)}$ fixes $\{L, L'\}$.  Thus $L$ and $L'$ are defined over $\kappa(T)\left(\sqrt{d}\right)$ for some $d$. If $d$ is a nonsquare, then we write $\sigma$ for the nontrivial element of $\Gal\left(\kappa(T)\left(\sqrt{d}\right)/\kappa(T)\right)$; note that $\sigma(L) = L'$.
		 
		 	Let $\ell,\tilde{\ell},\ell',\tilde{\ell'}$ be such that $L = V(\ell, \tilde{\ell})$ and $L' = V(\ell', \tilde\ell')$.  Since $\{P\} = L \cap L'$, the span of $\ell$ and $\tilde\ell$ has a one-dimensional subspace in common with the span of $\ell'$ and $\tilde\ell'$.  Without loss of generality, we may assume that $\ell = \tilde\ell'$ so $L' = V(\ell, \ell').$  If $d$ is a nonsquare, then since $\sigma(L) = L'$, $\sigma(\tilde \ell)$ must be in the span of $\ell$ and $\ell'$.  Further, since $L$ is not defined over $\kappa(T)$, $\sigma(\tilde \ell)$ is not a multiple of $\ell$. Therefore, we may assume that $\ell' = \sigma(\tilde \ell)$.
			
		 	Since $L$ lies on $V(Q_T)$, $Q_T$ must equal $\ell\cdot f + \tilde\ell \cdot \tilde f$ for some linear forms $f$ and $\tilde f$.  Similarly, since $L'$ lies on $V(Q_T)$, $Q_T$ must equal $\ell\cdot g + \ell'\cdot g'$ for some other linear forms $g$ and $g'$.  Therefore, modulo $\ell$, $cQ_T \equiv -\ell'\tilde\ell$ for some $c\in \kappa(T)^{\times}$. Since $\ell'\tilde\ell$ is Galois invariant, there exist linear forms $\ell_3,\ell_4$ defined over $\kappa(T)$ such that $\ell'\tilde\ell = \ell_3^2 - d\ell_4^2$.
		 
		 	To complete the proof, it remains to show that $\ell$ is a scalar multiple of a linear form defined over $\kappa(T)$.  This is clear if $d$ is a square.  Otherwise, $\ell$ is a linear combination of $\sigma(\ell)$ and $\sigma(\tilde\ell)$. Assume that this linear combination involves a nonzero multiple of $\sigma(\tilde\ell)$.  Then, by applying $\sigma$, we see that $\tilde\ell$ can also be expressed as a linear combination of $\ell$ and $\sigma(\ell)$.  This implies that $L = L'$, a contradiction.  Thus, $\ell$ must be a scalar multiple of $\sigma(\ell)$, say $\ell = \lambda \sigma(\ell)$ for some $\lambda\in \kappa(T)\left(\sqrt{d}\right)$.  Write $\ell$ as $f + \sqrt{d}g$, where $f$ and $g$ are linear forms over $\kappa(T)$.  After expanding the relation that $\ell = \lambda\sigma(\ell)$, we see that $f$ and $g$ are scalar multiples of each other, and thus that $\ell$ is a scalar multiple of a linear form defined over $\kappa(T)$. Since the discriminant of $\ell_1\ell_2 - \ell_3^2 + d\ell_4^2 = cQ_T$ is $d$, $d$ must differ from $\eps_T$ by a square.  This completes the proof.
		\end{proof}
		
		\subsection{Generators for the geometric Picard group}
		\label{subsec:Generators}\label{sec:PicGens}

			For $i\in\{0, \ldots, 4\}$, write $W_i := V(Q_i)$ and let $k_i$ be a finite extension of $k(t_i)$ such that
			\begin{equation*}\label{eq: ki}
				W_i\textup{ has a smooth }k_i\textup{-point},  
				\quad\textrm{and}\quad
				[k_i(\sqrt{\eps_{t_i}}):k_i] = 
				[k(t_i)(\sqrt{\eps_{t_i}}):k(t_i)].
			\end{equation*}
			For any smooth point $P_i\in W_i(k_i)$, let $H_{P_i}$ be the hyperplane  tangent to $W_i$ at $P_i$.  By Lemma~\ref{lem:Rank4NormalForm}, we have $W_i\cap H_{P_i} = L_{P_i}\cup L_{P_i}'$, where $L_{P_i}$ and $L_{P_i}'$ are planes defined over $k_i\left(\sqrt{\eps_{t_i}}\right)$.

			Define 
			\[
				C_{P_i}  = X \cap L_{P_i} \quad\textrm{and}\quad 
				C_{P_i}' = X \cap L_{P_i}'.
			\]
			As in~\cite[Theorem~4]{BBFL}, note that
			\[
				C_{P_i} = X \cap L_{P_i} = W \cap W_i \cap L_{P_i} = 
				W \cap L_{P_i},
			\]
			where $W$ is a quadric hypersurface associated to a smooth quadric in the pencil defined by $Q$ and $\widetilde{Q}$.  Since $W \cap L_{P_i}$ is the intersection of a smooth quadric 3-fold with a plane, the curve $C_{P_i}$ is a conic; the same is true for $C_{P_i}'$.  By construction, these conics are defined over the field $k_i\left(\sqrt{\eps_{t_i}}\right)$.  

			A different choice of smooth point $\tilde P_i \in W_i(k_i)$ gives rise to two different planes $L_{\tilde P_i}$ and $L_{\tilde P_i}'$, and hence two different conics $C_{\tilde P_i}$ and $C_{\tilde P_i}'$.  Since the planes $L_{\tilde P_i}$, $L_{\tilde P_i}'$, $L_{P_i}^{\phantom{'}}$ and $L_{P_i}'$ are contained in two pencils, without loss of generality, we may assume that $C_{\tilde P_i}\sim C_{P_i}$ and $C_{\tilde P_i}' \sim C_{P_i}'$.  In particular, for any element $\sigma \in G_{k(t_i,\sqrt{\eps_{t_i}})}$, we have $C_{\sigma(P_i)} = \sigma(C_{P_i})$.  Henceforth, we write $C_i$ and $C_i'$ for the classes in $\Pic \Xbar$ of $C_{P_i}$ and $C_{P_i'}$ respectively.  This discussion shows that $C_i$ and $C_i'$ are defined over $k(t_i, \sqrt{\eps_{t_i}}).$

			\begin{proposition}\label{prop:PicGens}
				After possibly interchanging $C_i$ and $C_i'$ for some indices $i$, we may assume that the group $\Pic(\Xbar) \cong \Z^6$ is freely generated by the classes of the following divisors
				\begin{equation}\label{eq: Picgens}
					\frac{1}{2}(H + C_0 + C_1 + C_2 + C_3 + C_4),\; 
					C_0,\; C_1,\; C_2,\; C_3, \textrm{and } C_4,
				\end{equation}
				where $H$ is a hyperplane section of $X$, and the $C_i$ are conics as above.
			\end{proposition}

			\begin{proof}
				Recall that, by the definition of a del Pezzo surface, $\Xbar$ is $\kbar$-isomorphic to the blow-up of $\PP^2$ at five points $\{p_0,\dots,p_4\}$, no three of which are colinear. For $0 \leq i \leq 4$ write $E_i$ for the exceptional divisor corresponding to $p_i$ under the blow-up map, and let $L$ be the strict transform of a line in $\PP^2$ that does not pass though any of the points $p_i$.

				Since $X$ is anticanonically embedded, we have $H \sim -K_X\sim 3L - \sum E_i$.  By~\cite[Theorems~2 and~4]{BBFL}, after possibly interchanging $C_i$ and $C_i'$, we may assume that $C_i \sim L - E_i$ and $C_i' \sim -K_X - C_i$. Thus			\[
					\frac12(H + C_0 + C_1 + \cdots C_4)\sim
					\frac12(3L - \sum E_i + \sum(L - E_i)) \sim 4L - \sum E_i 
					\in \Pic \Xbar.
				\]

				Using the above expressions for $C_i$ and $C_i'$ in terms of $L$ and $E_i$, we compute that
				\[
					C_i^2 = (C_i')^2 = 0,\quad C_i\cdot H = C_i\cdot C_i' = 2,
					\textrm{ and }C_i\cdot C_j = C_i'\cdot C_j' = C_i \cdot C_j' = 1 
					\textrm{ for }i\neq j.
				\]
				These intersection numbers imply that the divisor classes~\eqref{eq: Picgens} span a rank $6$ sublattice of $\Pic(\Xbar)$.  Furthermore, this sublattice is unimodular, so it must be equal to all of $\Pic(\Xbar)$.
			\end{proof}

		\subsection{Galois action on $\Pic(\Xbar)$}\label{subsec:Galois}
		\label{sec:PicGalois}

			Let $\Gamma$ be a graph on ten vertices indexed by the conics $C_0,C_0',\dots,C_4,C_4'$ whose only edges join $C_i$ to $C_i'$ for $i = 0,\dots,4$. Let $c_i \in \Aut(\Gamma)$ be the graph automorphism that exchanges $C_i$ with $C_i'$ and leaves every other vertex fixed. The group $\Aut(\Gamma)$ is a semi-direct product
			\[
				(\Z/2\Z)^5 \rtimes S_5,
			\]
			where the $S_5$ factor permutes the set of pairs $\left\{\{C_i,C_i'\} : 0\leq i \leq 4\right\}$, and  $c_i$ generates one of the $\Z/2\Z$-factors. By the discussion in~\cite[pp.\ 8--10]{KST-dP4s}, there is a natural embedding of index $2$
			\[
				O(K_X^{\perp}) \hookrightarrow \Aut(\Gamma).
			\]
			An element in $\Aut(\Gamma)$ is in the image of this map if and only if it is a product of an even number of $c_i$'s with an element of $S_5$. 

			The group $G_k$ acts on $\Pic(\Xbar)$, and it preserves the intersection pairing, as well as $K_X$. The action therefore factors through the group $O(K_X^{\perp})$, which we consider in turn as a subgroup of $\Aut(\Gamma)$, as above. Projection onto the $S_5$-factor gives the usual action of $G_k$ on the set $\scrS(\kbar) = \{t_0,\dots, t_4\}$.

			For any $T \in \scrS$, let $\Gamma_T$ be the graph with $2\deg(T)$ vertices indexed by $C_i$ and $C_i'$ for $t_i\in T(\kbar)$, whose only edges join $C_i$ and $C_i'$ for $t_i\in T(\kbar)$. Note that $\Gamma_T$ is a subgraph of $\Gamma$, and that $\Gamma$ is the union of $\Gamma_T$ over all $T\in\scrS$.  There is a natural injective map
			\[
				\prod_{T\in \scrS} \Aut(\Gamma_T) \hookrightarrow \Aut(\Gamma),
			\]
			which we use to identify the domain with its image in $\Aut(\Gamma)$.

			\begin{proposition}\label{prop:GalAction}
				The action of the group $G_k$ on $\Pic(\Xbar)$ induces an action on $\Gamma$ that factors through 
				\[
					\left(\prod_{T \in \scrS} \Aut(\Gamma_T) \cap O(K_X^{\perp})\right)\subseteq \Aut(\Gamma).
				\]
				Moreover, $G_k$ acts transitively on the set $\{C_i,C_i': t_i \in T(\kbar)\}$ if and only if $\eps_T \notin \kappa(T)^{\times 2}$.
			\end{proposition} 

			\begin{proof}
				The first part of the proposition follows from the above discussion.  For the last claim, it suffices to note that for $t_i \in T(\kbar)$, the pair of divisor classes $\{C_i,C_i'\}$ is defined over $k(t_i)$, whilst the constituent classes of the pair are defined over $k\left(t_i,\sqrt{\eps_{t_i}}\right)$.
			\end{proof}

	\section{All Brauer elements are vertical}\label{sec:VerticalElements}

		We keep the notation from the previous section. Throughout, we assume that for all $T\in\scrS$, the quadric hypersurface $V(Q_T)$ has a smooth $\kappa(T)$-point.

		\subsection{The construction}\label{subsec:Construction}

			Assume that there exists a $k$-rational subscheme $\scrT\subseteq\scrS$ such that
		\begin{equation}\label{eq:Assumption}
			\deg(\scrT) = 2, \quad
			\prod_{T\in\scrT}\Norm_{\kappa(T)/k}(\eps_T) \in k^{\times2},\quad
			\textup{and }\eps_T\not\in\kappa(T)^{\times2}\textup{ for all }T\in\scrT.
			\tag{$\star$}
		\end{equation}

			\begin{lemma}\label{lem:DefnOfLScrT}
				If $\scrT$ satisfies~\eqref{eq:Assumption}, then $\eps_T$ lies in the image of $k^{\times}/k^{\times2}\to\kappa(T)^{\times}/\kappa(T)^{\times2}$ for all $T\in\scrT$.
			\end{lemma}	
			\begin{proof}
				If $\scrT$ contains a degree $1$ point, then the claim is immediate as $\kappa(T) = k$ for all $T\in\scrT$.  Assume that $T\in\scrT$ is a degree $2$ point; write $T(\kbar) = \{t, t'\}.$  By the hypothesis, $\eps_t\eps_{t'}\in k^{\times2}$, and so $\sqrt{\eps_{t'}} = c/\sqrt{\eps_t}$ for some $c\in k^{\times}.$  

				Let $K := k(t,\sqrt{\eps_t})$; this is a degree $4$ Galois extension of $k$ and any element of $\Gal(K/k)$ sends $\sqrt{\eps_t}$ to an element in $\{\sqrt{\eps_{t}}, -\sqrt{\eps_t}, c/\sqrt{\eps_{t}}, -c/\sqrt{\eps_{t}}\}$.  Using this, one can check that there is no order $4$ automorphism of $K/k$.  Thus $\Gal(K/k) = (\Z/2\Z)^2$ and $K$ is a biquadratic extension of $k$ that contains $k(t)$.  A calculation then shows that $\eps_t\in \{a\alpha^2, \alpha^2 : a\in k^{\times}, \alpha\in k(t)^{\times}\}$.  By~\eqref{eq:Assumption}, $\eps_t\not\in k(t)^{\times2}$, which proves the lemma.
			\end{proof}

			If $\scrT$ satisfies~\eqref{eq:Assumption}, then the lemma allows us to assume that $\eps_T \in k^\times$ for all $T \in \scrT$. Define $L_\scrT = k(\sqrt{\eps_T})$ for some $T \in \scrT$; this extension is a quadratic extension, independent of the choice of $T$.  

			By assumption, for all $T\in\scrT$, $V(Q_T)$ has a smooth $\kappa(T)$-point.  Let $\ell_T$ be a $\kappa(T)$-linear form such that the associated hyperplane is tangent to $V(Q_T)$ at a smooth point.

			\begin{lemma}\label{lem:DefnOfCalAScrT}
				Assume that $\scrT$ satisfies~\eqref{eq:Assumption}, and let $\ell$ be any $k$-linear form. Then the cyclic algebra 
				\[
					\calA_{\scrT}:=\left(L_{\scrT}/k, 
					\ell^{-2}\prod_{T\in\scrT}\Norm_{\kappa(T)/k}(\ell_T)\right)
				\]
				is unramified and thus is in the image of $\Br X\to\Br\kk(X)$.
			\end{lemma}
			\begin{proof}
				By the purity theorem~\cite[\S\S6,7]{Grothendieck-BrauerIII}, the class of a cyclic algebra $(L/k, f)$ is in the image of $\Br X \to \Br \kk(X)$ if and only if for every prime divisor $V\subseteq X$ we have $L\subseteq\kappa(V)$ whenever $\ord_V(f)\equiv 1\bmod 2$.  Therefore, we must show that $L_{\scrT} \subseteq \kappa(C_T\cup C'_T)$ for all $T\in\scrT$. This containment holds because $C_T$ and $C_T'$ are conjugate over $L_\scrT$.
			\end{proof}

		\subsection{Nontriviality of $\calA_{\scrT}$}
		\label{subsec:Nontriviality}

			We would like to determine when $\calA_{\scrT}$ is nontrivial, i.e., when $\calA_{\scrT} \notin \Br_0 X$.  Let $L$ be a cyclic extension of $k$, and let $f \in \kk(X)^\times$. Define
			\[
				\Br_{\textup{cyc}}(X,L) = \left\{
				\begin{array}{c}
					\textup{classes $[(L/k,f)]$ in the image of the } \\
					\textup{map }\Br X/\Br_0 X \to \Br\kk(X)/\Br_0 X
				\end{array}
				\right\}
			\]
			Fix a generator $\sigma$ of $\Gal(L_{\scrT}/k)$. We view $\Norm_{L_\scrT/k}$ and $1 - \sigma$ as endomorphisms of $\Pic(X_L)$, and we consider the image of $\calA_{\scrT}$ under the composition of the following maps
			\begin{equation}\label{eq:Maps}
				\Br_{\textup{cyc}}(X, L_\scrT) \stackrel{\sim}{\longrightarrow}
				\frac{\ker \Norm_{L_\scrT/k}}{\im (1 - \sigma)} \hookrightarrow
				\HH^1(G_k, \Pic(\Xbar))[2] \stackrel{\sim}{\longrightarrow}
				\frac{\left(\Pic\Xbar/2\Pic\Xbar\right)^{G_k}}
				{\left((\Pic\Xbar)^{G_k}/2(\Pic\Xbar)^{G_k}\right)}.
			\end{equation}
			The first two maps are described in~\cite[Thm 3.3]{VA-WeakApprox}, and the last map is obtained by considering the long exact sequence in cohomology associated to the short exact sequence $0 \to \Pic\Xbar \stackrel{\times2}{\to}\Pic\Xbar \to \left(\Pic\Xbar/2\Pic\Xbar\right) \to 0$. (In~\cite[\S3]{VA-WeakApprox} $k$ is assumed to be a number field; however, this hypothesis is not necessary for the result we cite.)

			By~\cite[Thm 3.3]{VA-WeakApprox}, a degree $2$ cyclic algebra $(L/k, f)$ maps to a cocycle in $\HH^1(G_k, \Pic\Xbar)$ that sends $G_{L}$ to the identity, and that sends any element outside of $G_{L}$ to a divisor $D$ such that $\divv(f) = \sum_{\sigma \in \Gal({L/k})} \sigma(D)$. Therefore, the cyclic algebra $\calA_{\scrT}$ maps to the cocycle
			\begin{equation}\label{eq:Cocycle}
				\sigma \mapsto 
				\begin{cases}
					-H + \sum_{t_i\in\scrT(\kbar)} C_i & \textup{if }\sigma\not\in G_{L_\scrT},\\
					0 & \textup{otherwise},
				\end{cases}
			\end{equation}
			in $\HH^1(G_k, \Pic \Xbar).$  Since $\sum_{t_i\in\scrT(\kbar)} C_i$ is fixed by every element of $G_{L_\scrT}$ and 
			\[
				\frac12\left(\sum_{t_i\in\scrT(\kbar)} C_i - 
				\sigma\bigg(\sum_{t_i\in\scrT(\kbar)} C_i\bigg)\right) = 
				\frac12\left(\sum_{t_i\in\scrT(\kbar)} C_i
				- \sum_{t_i\in\scrT(\kbar)} C'_i\right)
				= -H + \sum_{t_i\in\scrT(\kbar)} C_i,
			\]
			for all $\sigma \in G_k\setminus G_{L_\scrT}$, the cocycle~\eqref{eq:Cocycle} maps to $\sum_{t_i\in\scrT(\kbar)} C_i$ under the last map in~\eqref{eq:Maps}.  Thus $\calA_{\scrT}$ is trivial if and only if $\sum_{t_i\in\scrT(\kbar)} C_i \in 2\Pic\Xbar + (\Pic \Xbar)^{G_k}$.

			\begin{proposition}\label{prop:Nontriviality}
				If $\scrT$ is a $k$-subscheme of $\scrS$ that satisfies~\eqref{eq:Assumption}, then $\calA_{\scrT} \not\in \Br_0(X)$ if and only if there exists $T\in\scrS\setminus\scrT$ such that $\varepsilon_{T}\not\in\kappa(T)^{\times2}$.
			\end{proposition}

			\begin{proof}
				As noted above, $\calA_{\scrT}$ is trivial if and only if $\sum_{t_i\in\scrT(\kbar)} C_i \in 2\Pic\Xbar + (\Pic\Xbar)^{G_k}$.  By Proposition~\ref{prop:PicGens}  the Picard group is freely generated by
				\[
					\frac12\left(H + C_0 + C_1 + C_2 + C_3 + C_4\right), \; C_0, \; C_1,\;
					C_2,\; C_3, \;\textup{and } C_4.
				\]
				Since, by~\eqref{eq:Assumption}, $\eps_T \notin k(T)^{\times 2}$ for all $T \in \scrT$, Proposition~\ref{prop:GalAction} implies that $\sum_{t_i\in\scrT(\kbar)} C_i$ is not fixed by Galois. Therefore, $\sum_{t_i\in\scrT(\kbar)} C_i \in 2\Pic\Xbar + (\Pic\Xbar)^{G_k}$ if and only if there is some choice of sign such that
				\[
					\sum_{t_i\in(\scrS\setminus\scrT)(\kbar)} \pm C_{i} 
					\in (\Pic \Xbar)^{G_k}.
				\]

				Assume that there exists $T\in\scrS\setminus\scrT$ with $\varepsilon_{T}\not\in\kappa(T)^{\times2}$.  Then, by Proposition~\ref{prop:GalAction}, the absolute Galois group acts transitively on $\{C_i, C'_i : t_i\in T(\kbar)\}$.  Thus, for some $t_j\in(\scrS\setminus\scrT)(\kbar)$, there exists an element $\sigma\in G_k$ such that 
				\[
					\sigma\left(\pm C_{j} + \sum_{t_i\in(\scrS\setminus\scrT)(\kbar), i\ne j}\pm C_i\right) = 
					\mp C_{j} + D,
				\]
				where $D$ is a divisor that does not involve $C_{j}$.  Therefore, for any choice of signs,
				\[
					\sum_{t_i\in(\scrS\setminus\scrT)(\kbar)} \pm C_{i} 
					\notin (\Pic \Xbar)^{\sigma} \supseteq (\Pic \Xbar)^{G_k},
				\]
				and so $\calA_{\scrT}$ is nontrivial in $\Br X/\Br_0 X$.

				Now assume that for all $T\in \scrS\setminus\scrT$, $\eps_T\in\kappa(T)^{\times2}$.  By Proposition~\ref{prop:GalAction}, $G_k$ does not act transitively on $\{C_i, C_i' : t_i\in T(\kbar)\}$.  Since $G_k$ does act transitively on $\{t_i : t_i\in T(\kbar)\}$, the set $\{C_i, C_i' : t_i\in T(\kbar)\}$ breaks up into two Galois invariant sets of equal size such that for all $t_i\in T(\kbar)$, $C_i$ and $C_i'$ are not contained in the same set.  

				Without loss of generality, we may assume that $(\scrS\setminus\scrT)(\kbar) = \{t_0,t_1,t_2\}$.  The previous discussion shows that, possibly after renumbering, at least one of
				\[
					\{C_{0}, C_{1}, C_{2}\} \quad \textup{or}\quad
					\{C_{0}, C_{1}, C'_{2}\}
				\]
				are Galois invariant.  If the first set is Galois invariant, then $C_{0} + C_{1} + C_{2}\in (\Pic \Xbar)^{G_k}$ and if the latter is Galois invariant, then $C_{0} + C_{1} - C_{2} \in(\Pic\Xbar)^{G_k}.$  In either case, this shows that $\calA_{\scrT}$ is in $\Br_0 X$.
			\end{proof}

		\subsection{Main result}\label{subsec:MainResult}

			\begin{thm}\label{thm:MainExplicit}
				If $\Br X \neq \Br_0 X$ then there exists $\scrT \subseteq \scrS$ satisfying~\eqref{eq:Assumption} and $T \in \scrS\setminus\scrT$ such that $\eps_T \notin \kappa(T)^{\times 2}$. Furthermore, if $\#(\Br X/\Br_0 X)= 4$ then there exist $T_0, T_1, T_2\in\scrS(k)$ such that every pair satisfies~\eqref{eq:Assumption}.
				
				{If, for all $T\in\scrS$ of degree at most $2$, the quadric hypersurface $V(Q_T)$ has a smooth $\kappa(T)$-point, then the converses of the above two statements hold. Moreover, }any nontrivial element in $\Br X/\Br_0 X$ is of the form $\calA_{\scrT}$ for some $\scrT\subseteq\scrS$, and if $\#(\Br X/\Br_0 X)= 4$, then $\Br X/\Br_0 X = \{\id, \calA_{\{T_0, T_1\}}, \calA_{\{T_0, T_2\}}, \calA_{\{T_1, T_2\}}\}$.
			\end{thm}
			
			\begin{proof}
				From the Hochschild-Serre spectral sequence, we have an injection
				\[
					\frac{\Br X}{\Br_0 X} \hookrightarrow \HH^1(G_k, \Pic \Xbar);
				\]
				\emph{a priori}, this map need not be surjective for an arbitrary field $k$ of characteristic not $2$. Using Proposition~\ref{prop:GalAction} and a case by case analysis, one can compute that $\#\HH^1(G_k, \Pic \Xbar) \mid 4$, and that if $\#\HH^1(G_k, \Pic \Xbar) = 2$, then there exists $\scrT \subset \scrS$ satisfying~\eqref{eq:Assumption}, and a $T'\in \scrS\setminus \scrT$ such that $\eps_{T'}\notin\kappa(T')^{\times2}$.  By assumption, $V(Q_T)$ has a smooth $\kappa(T)$-point for all $T\in\scrT$, so $\calA_\scrT$ exists, and by Proposition~\ref{prop:Nontriviality} it is a nontrivial element of $\Br X/\Br_0 X$. Since $\#(\Br X/\Br_0 X)\leq 2$, $\calA_\scrT$ must be the unique nontrivial element.

				Furthermore, we compute that if $\#\HH^1(G_k, \Pic \Xbar) = 4$, then there exist $T_0, T_1$, and $T_2$ in $\scrS(k)$ such that every pair $\{T_i,T_j\}$ satisfies~\eqref{eq:Assumption}.  {As above, we see that} 
				\[
					\calA_{\{T_0,T_1\}}, \quad \calA_{\{T_0,T_2\}},\quad
					\textup{and}\quad \calA_{\{T_1,T_2\}}
				\]
				{exist and }are nontrivial elements of $\Br X/\Br_0 X$.  Moreover, from the construction it is clear that $\calA_{\{T_1,T_2\}} = \calA_{\{T_0,T_1\}}+ \calA_{\{T_0,T_2\}}$, so  $\#(\Br X/\Br_0 X) = 4$ and these elements are the three nontrivial elements in $\Br X/\Br_0 X$.  A \texttt{Magma}~\cite{Magma} script to verify these claims can be found in the \texttt{arXiv} distribution of this article.  {Alternatively, the industrious reader may carry out these computations by hand using arguments similar to those in the proof of Proposition~\ref{prop:Nontriviality}.}
				\end{proof}

				When $k$ is a number field, work of Wittenberg enables us to give an alternative proof~\cite{Wittenberg-SpringerThesis}.  

				\begin{proof}[Proof of Theorem~\ref{thm:MainExplicit} assuming that $k$ is a number field]
					Let $\scrS_1 \subseteq \scrS$ be the set of $T \in \scrS$ such that $\eps_T \notin \kappa(T)^{\times2}$, and let $\scrS_0\subseteq \scrS_1$ be a minimal generating set for 
					\[
						\langle\Norm_{\kappa(T)/k}(\varepsilon_T) : T\in \scrS\rangle 
						\subseteq k^{\times}/k^{\times2}.
					\]
					By~\cite[Thm. 3.37]{Wittenberg-SpringerThesis}, $\#(\Br X/\Br_0 X) = 2^{\max(0, n - d - 1)}$ where $d = \#\scrS_0$ and $n = \#\scrS_1$, so, hereafter, we assume $n-d\geq 2$. Throughout the proof, we will often rely on~\cite[Prop. 3.39]{Wittenberg-SpringerThesis}, which states that 
				\begin{equation}\label{eq:product}
					\prod_{T\in\scrS}\Norm_{\kappa(T)/k}(\eps_T) \in k^{\times2}.
				\end{equation}
				Also, when discussing multiplicative independence of elements of the form $\Norm_{\kappa(T)/k}(\eps_T)$ below, we are implicitly working in $k^{\times}/k^{\times2}$.

				We will show that if $n - d = 2$, then there exists $\scrT \subset \scrS$ satisfying~\eqref{eq:Assumption}, and for which $\scrS_1\setminus \scrT \neq \emptyset$. By Proposition~\ref{prop:Nontriviality}, this implies that $\calA_\scrT$ is the unique nontrivial element of $\Br X/\Br_0 X$.  We will also show that $n - d\leq 3$ and that if $n - d  = 3$, then there exist three points $T_0, T_1$, and $T_2$ of $\scrS(k)$ such that every pair satisfies~\eqref{eq:Assumption}. Then, by Proposition~\ref{prop:Nontriviality}, $\calA_{\{T_0,T_1\}}, \calA_{\{T_0,T_2\}}, \calA_{\{T_1,T_2\}}$ are nontrivial elements of $\Br X/\Br_0 X$.  Moreover, from the construction it is clear that $\calA_{\{T_1,T_2\}} = \calA_{\{T_0,T_1\}}+ \calA_{\{T_0,T_2\}}$, so these elements are the three nontrivial elements in $\Br X/\Br_0 X$.  These two claims will complete the proof of the theorem.

				Consider the case when $n = 2$. Then we must have $d = 0$, which means that $\scrS_1$ must consist entirely of points of degree $> 1$. Since $\scrS$ has total degree 5, there is a degree $2$ point $T\in \scrS$. Take $\scrT := \{T\}$; note that it satisfies~\eqref{eq:Assumption}, and that $\scrS_1\setminus \scrT \neq \emptyset$.

				Henceforth, we assume that $n > 2$.  This means that for any degree $2$ subscheme $\scrT \subset \scrS_1$, we have $\scrS_1\setminus \scrT \neq \emptyset$. Suppose that $n = 3$.  Then $\scrS$ contains at least $(\#\scrS - 2)$ $k$-points. This implies that there is a $T_0 \in \scrS(k)$ with $\eps(T_0) \notin \kappa(T_0)^{\times2}$, which in turn shows that $d = 1$, because $n - d \geq 2$. Without loss of generality, take $\scrS_0 = \{T_0\}$. Then either there is a $T_1 \in \scrS(k)$ with $\eps_{T_1} \notin \kappa(T_1)^{\times2}$, in which case the set $\scrT := \{T_0,T_1\}$ satisfies~\eqref{eq:Assumption}, or the degree sequence of $\scrS$ is $(1,2,2)$.  In the latter case, the relation~\eqref{eq:product} implies that exactly one of the two degree $2$ points $T \in \scrS$ satisfies $\Norm_{\kappa(T)/k}(\eps_T) \in k^{\times2}$. Then $\scrT := \{T\}$ satisfies~\eqref{eq:Assumption}. 

				Next, assume that $n = 4$ and $d \leq 2$. Then there exist at least three degree $1$ points $T_0$, $T_1$ and $T_2$ in $\scrS_1$, so, in particular, $d > 0$. If $d = 1$, then every order $2$ subset of $\{T_0,T_1,T_2\}$ satisfies~\eqref{eq:Assumption}. If $d = 2$, then either every order $2$ subset of $\{T_0,T_1,T_2\}$ is multiplicatively independent, in which case $\scrT := \scrS \setminus \{T_0,T_1,T_2\}$ satisfies~\eqref{eq:Assumption}, or there is exactly one order $2$ subset of $\{T_0,T_1,T_2\}$ that is multiplicatively dependent. In the latter case, this unique subset satisfies~\eqref{eq:Assumption}.

				It remains to consider the case when $n = 5$ and $d \leq 3$. By~\eqref{eq:product}, we have $d \geq 2$. Write $\scrS = \{t_0,t_1,t_2,t_3,t_4\}$. Without loss of generality, assume that $\eps_{t_0}$ and $\eps_{t_1}$ are multiplicatively independent. If $d = 2$, then, up to renumbering, we have $\eps_{t_4} = \eps_{t_0}\eps_{t_1}$, and either
				\[
					\eps_{t_2} = \eps_{t_3} = \eps_{t_4} \quad\textup{or}\quad
					\eps_{t_0} = \eps_{t_2} = \eps_{t_3}.
				\]
				In either case, we obtain our desired result. If $d = 3$, then we may assume that $\eps_{t_0}$, $\eps_{t_1}$ and $\eps_{t_2}$ are multiplicatively independent, that $\eps_{t_3} = \eps_{t_0}\eps_{t_1}$, and that $\eps_{t_2} = \eps_{t_4}$. Then $\scrT := \{t_2,t_4\}$ satisfies~\eqref{eq:Assumption}.
			\end{proof}	

			\begin{cor}\label{cor:Vertical}
				Assume that for all {$T\in\scrS(k)$}, the quadric hypersurface $V(Q_T)$ has a smooth {$k$}-point. Then every element of $\Br X/\Br_0 X$ is vertical for a genus $1$ fibration $X\dasharrow \PP^1$ that has at most $2$ geometrically reducible fibers.
			\end{cor}
			\begin{proof}
				Let $\calA$ be a nontrivial element of $\Br X/\Br_0 X$. By Theorem~\ref{thm:MainExplicit}, $\calA = \calA_{\scrT}$ for some $\scrT\subseteq\scrS$ satisfying~\eqref{eq:Assumption}.  If $\scrT = \{T_0,T_1\}$, let $\ell_i$ be a $k$-linear form tangent to $V(Q_{T_i})$ at a smooth point for $i = 0, 1$. Otherwise, $\scrT$ consists of a degree $2$ point $T$; in this case let $d$ be the discriminant of the residue field of $T$, and let $\ell_0$ and $\ell_1$ be $k$-linear forms such that $\ell_0 + \sqrt{d}\ell_1$ is tangent to $V(Q_T)$ at a smooth point.

				Define a map $f\colon X\dasharrow\PP^1$ by $x \mapsto [\ell_0(x):\ell_1(x)]$. Since $\calA_{\scrT}$ is equal to either 
				\[
					\left(L_{\scrT}/k, \frac{\ell_0}{\ell_1}\right)\quad\textup{or}\quad
					\left(L_{\scrT}/k, \frac{\ell_0^2 - d\ell_1^2}{\ell_1^2}\right),
				\]
				it is clear that $\calA_{\scrT}\in\Br_{\verti}^{(f)} X$. Moreover, every fiber of $f$ is an intersection of a hyperplane with $X$, i.e., a genus $1$ curve.

				It remains to prove that $f$ has at most $2$ reducible fibers.  Since the $\kappa(T)$-points on $V(Q_T)$ are Zariski dense, it suffices to show that the locus of $\{\ell_0, \ell_1\}$ that give rise to a rational map with at least $3$ geometrically reduced fibers is a proper closed subset of $\prod_{T \in \scrT}R_{\kappa(T)/k}\left(V(Q_T)\right)$, where $R_{\kappa(T)/k}(-)$ denotes Weil restriction of scalars.  Thus, we may pass to the separable closure $\kbar$.  

				First consider the case when $\#\scrT = 2$.  After a change of variables, we may assume that $\{Q_t : t\in \scrT(\kbar)\}$ equals $\{Q_1,Q_2\}$ where
				\[
					Q_1 := x_0^2 - x_2^2 - x_3^2 + \gamma^2x_4^2, \quad 
					Q_2 := x_1^2 - \alpha^2x_2^2 + \beta^2x_3^2 - x_4^2,
				\]
				for some $\alpha,\beta,\gamma\in\kbar^{\times}.$  The locus of maps with at least $3$ geometrically reduced fibers is closed; we must show that it is proper.  Consider the linear forms
				\[
					\ell_1 := x_0 - x_2 - \gamma x_3 + \gamma^2x_4, \quad 
					\ell_2 := x_1 - \alpha x_2 + \beta^2\alpha^{-1}x_3 - \beta\alpha^{-1}x_4.
				\]
				A computation shows that $V(\ell_1)$ and $V(\ell_2)$ are respectively tangent to $V(Q_1)$ and $V(Q_2)$ at smooth points, and that projection from $V(\ell_1,\ell_2)$ gives a map with at most $2$ geometrically reduced fibers.\footnote{A \texttt{Magma}~\cite{Magma} script to verify this claim and the similar claim below can be found in the \texttt{arXiv} distribution of this article.}
				  
				Now consider the case when $\scrT$ consists of a single degree $2$ point $T$.  Then 
				\[
					\prod_{T \in \scrT}R_{\kappa(T)/k}\left(V(Q_T)\right)(\kbar) 
					= \left(V(Q_t)(\kbar)\right)^2
				\]
				for some $t\in T(\kbar)$.  After a change of variables, we may assume that $Q_t =Q_1$ above, and that $Q_2$ is some other quadric in the pencil.  Consider the linear forms
				\[
					\ell_1 := x_0 - x_2 - \gamma x_3 + \gamma^2x_4, \quad 
					\ell_2 := x_0 - \gamma x_2 - x_3 + \gamma^2x_4.
				\]
				A computation shows that $V(\ell_1)$ and $V(\ell_2)$ are both tangent to $V(Q_1)$ at smooth points, and that projection from $V(\ell_1,\ell_2)$ gives a map with at most $2$ geometrically reduced fibers.  This completes the proof.
			\end{proof}
			\begin{cor}
				Assume that for all $T\in\scrS$ of degree at most $2$, the quadric hypersurface $V(Q_T)$ has a smooth $\kappa(T)$-point.  If $\#(\Br X/\Br_0 X) = 4$, then every nontrivial element is cyclic for the same quadratic extension.
			\end{cor}
			\begin{proof}
				It follows from the definition of $L_{\scrT}$ given in Lemma~\ref{lem:DefnOfCalAScrT} that $L_{\{T_0,T_1\}} = L_{\{T_0,T_2\}}= L_{\{T_1,T_2\}}$.
			\end{proof}

	\section{Computational applications}\label{sec:Computations}

		\subsection{A practical algorithm for computing $\Br X/\Br_0 X$}
		\label{subsec:Algorithm}
		
			The construction given in \S\ref{subsec:Construction}, together with Theorem~\ref{thm:MainExplicit}, yields a practical algorithm for the computation of $\Br X/\Br_0 X$, provided that for all $T\in\scrS$, the quadric hypersurface $V(Q_T)$ has a smooth $\kappa(T)$-point. (This hypothesis is satisfied, for example, whenever $k$ is a global field of characteristic not $2$ and $X(\A_k) \neq \emptyset$; see Lemma~\ref{lem:Rank4Solvable}.)  When $k$ is a number field, Bright, Bruin, Flynn, and Logan have implemented a different algorithm to compute $\Br X/\Br_0 X$~\cite{BBFL}.  Our algorithm has a simple implementation and in practice seems to be competitive with ~\cite{BBFL}. 

			The algorithm takes as input two quadrics $Q$ and $\widetilde{Q}$ with coefficients in $k$ such that $X = V(Q) \cap V(\widetilde{Q})$, and gives as output a complete list of the elements of $\Br X/\Br_0 X$.

			\begin{enumerate}
				\item Let $M$ and $\widetilde{M}$ be the symmetric matrices associated to the defining quadrics $Q$ and $\widetilde{Q}$, respectively. Compute the \defi{characteristic polynomial} of $X$
				\[
					f(\lambda, \mu) := \det(\lambda M + \mu\widetilde{M}).
				\]

				\item If $f(\lambda, \mu)$ is irreducible or has an irreducible quartic factor, then $\Br X = \Br_0 X$, so give as output $\{\textup{id}\}$. Otherwise, for a point $T\in \scrS := V(f) \subseteq\PP^1$, let $H_T$ be a hyperplane that does not contain the vertex of $Q_T$ and let $\eps_T$ be the discriminant of $H_T\cap V(Q_T)$.

				\item If there exist three degree $1$ points $T_0, T_1, T_2\in \scrS$ such that $\eps_{T_i}\notin k^{\times2}$ and $\eps_{T_i}\eps_{T_j} \in k^{\times2}$ for all $i,j\in\{0,1,2\}$, then choose points $P_i\in V^{\textup{smooth}}(Q_{T_i})(k)$.  Let $\ell_i$ be the linear form defining the tangent plane to $V(Q_{T_i})$ at $P_i$. Give as output
				\[
					\frac{\Br X}{\Br_0 X} = \left\{
					\textup{id}, \left(\eps_{T_0}, \ell_0\ell_1^{-1}\right),
					\left(\eps_{T_0}, \ell_0\ell_2^{-1}\right),
					\left(\eps_{T_0}, \ell_2\ell_1^{-1}\right)
					\right\}.
				\]

				\item If we get to this step, then $\#(\Br X/\Br_0 X) \leq 2$. If there is a $k$-subscheme $\scrT\subseteq \scrS$ that satisfies~\eqref{eq:Assumption} and a point $T'\in\scrS\setminus\scrT$ such that $\eps_{T'}\notin\kappa(T')^{\times2},$ then choose points $P_T\in V^{\textup{smooth}}(Q_{T})(\kappa(T))$ for all $T\in\scrT$. Let $\ell_T$ be the linear form defining the tangent plane to  $V(Q_{T})$ at $P_T$ and let $\ell$ be any other linear form.  Give as output
				\[
					\frac{\Br X}{\Br_0 X} = 
					\left\{
						\textup{id}, \left(\eps_{T}, \ell^{-2}
						\prod_{T\in\scrT}\Norm_{\kappa(T)/k}(\ell_T)\right)
					\right\}
				\]  

				\item If we get to this step, then $\Br X = \Br_0 X$, so give as output $\{\textup{id}\}$.
			\end{enumerate}
		
			\begin{example}
				Consider the degree $4$ del Pezzo surface $X$ over $\Q$ given by
				\begin{align*}
					x_0x_1 + x_2x_3 + x_4^2 &= 0, \\
					- x_0^2 - 3x_1^2 + x_2^2 - x_2x_3 + 2x_3^2 + 2x_3x_4 &= 0.
				\end{align*}
				The characteristic polynomial is $f(\lambda,\mu) := 2(\lambda^2 - 12\mu^2)(\lambda^3 - 2\lambda^2\mu - 7\lambda\mu^2 + 4\mu^3)$. Thus $\scrS = V(f) \subseteq \PP^1$ consists of a degree $2$ point $T$ and a degree $3$ point $T'$. We may take $H_T = V(x_0)$ and $H_{T'} = V(x_2)$; then $\eps_T = 120\sqrt{3} - 240 = - 5\cdot(2\sqrt{3} - 6)^2$ and $\eps_{T'} = -12\alpha^2 + 40\alpha - 16$, where $\alpha^3 - 2\alpha^2 - 7\alpha + 4 = 0$. Thus, the scheme $\scrT := \{T\}$ satisfies \eqref{eq:Assumption}, and $T' \in \scrS\setminus\scrT$ satisfies $\eps_{T'} \notin \kappa(T')^{\times 2}$. Furthermore, $\scrS(\Q) = \emptyset$, so $\Br X/\Br_0 X = \Z/2\Z$, generated by the class of $\calA_\scrT$. Let $P_T = [4:\sqrt{3}:1:0:0] \in V^{\textup{smooth}}(Q_T)\left(\Q(\sqrt{3})\right)$; then $\ell_T = 2x_0 + 2x_2 - x_3 - 2\sqrt{3}(x_1 - x_3)$, and finally
				\[
					\calA_\scrT = \left( -5, \frac{(2x_0 + 2x_2 - x_3)^2 - 12(x_1 - x_3)^2}{x_0^2} \right).
				\]
			\end{example}

		\subsection{Parameter spaces}\label{subsec:ParameterSpaces}

			The results in \S\ref{sec:VerticalElements} enable us to easily write down parameter spaces whose general points give a del Pezzo surface of degree $4$ with a nonconstant Brauer class. Moreover, the largest dimensional such space parametrizes every surface with a nonconstant Brauer class.
			
			For a general degree $4$ del Pezzo surface $X$ with a nonconstant Brauer class, the degeneracy locus $\scrS$ consists of a degree $2$ point $T$ and a degree $3$ point $T'$. The following proposition shows that the symmetric matrices associated to $X$ can be given in block form, where the sizes of the blocks correspond to the degrees of the points of $\scrS$. Thus, we may construct the aforementioned parameter space by first starting with a subvariety of $\left(\PP^{3 + 6 - 1}\right)^2$, where each coordinate corresponds to a coefficient of one of the defining quadrics of $X$.
			
			\begin{prop}
				\label{prop:block diagonal}
				Let $Q_1$ and $Q_2$ be quadrics in $n + 1$ variables over a field $k$ of characteristic not $2$, with $M_1$ and $M_2$ their respective associated symmetric matrices. Let $\{m_0, m_1, \ldots, m_r\}$ be the degrees of the closed points in $V\left(\det(\lambda M_1 + \mu M_2)\right)\subseteq \PP^1$.  Assume that $V(Q_1,Q_2) \subseteq \PP^n$ is smooth.  Then there is a change of variables, defined over $k$, such that $M_1$ and $M_2$ can be written in block diagonal form, where the blocks have sizes $\{m_0, m_1, \ldots, m_r\}$.
			\end{prop}
			
			\begin{remark}
				The case where $m_j = 1$ for all $j$ already appears in~\cite[Prop. 3.28]{Wittenberg-SpringerThesis}; the argument we give here has the same structure.
			\end{remark}
			\begin{proof}
				Let $T_j$ denote the closed point in the degeneracy locus of degree $m_j$ and let $v_{i,j}\in\kbar^n$ denote the kernels of the corresponding quadrics for $i = 0, \ldots, m_j-1$.  By~\cite[Prop. 3.28]{Wittenberg-SpringerThesis}, $\{v_{i,j} : j = 0, \ldots, r, \; i = 0, \ldots, m_j-1 \}$ is an orthogonal basis under $Q_1$ and $Q_2$.  Let $\theta_j$ be a $k$-algebra generator for $\kappa(T_j)$, let $\tilde\kappa(T_j)$ denote the Galois closure of $\kappa(T_j)$, and let $\sigma_j\in \Gal({\tilde\kappa(T_j)/k})$ be an order $m_j$ element.  Possibly after renumbering, we may assume that $\sigma_j^i(v_{0,j}) = v_{i,j}$.  We define
				\[
					w_{i,j} := \sum_{\ell = 0}^{m_j - 1} \sigma_j^{\ell}(\theta^i_j) v_{\ell, j}.
				\] 
				It is easy to check that $w_{i,j} \in k^n$, that $w_{i,j}$ form a basis, and that if $j \ne j'$, then $w_{i,j}$ and $w_{i',j'}$ are orthogonal under $Q_1$ and $Q_2$.  Then, the change of basis matrix that takes $\left\{w_{i,j}\right\}_{i,j}$ to the standard basis of $k^n$ transforms $M_1$ and $M_2$ into block diagonal form.
			\end{proof}

			We think of a point in $({\bf a},{\bf b}) \in \left(\PP^{3 + 6 - 1}\right)^2$ as two symmetric matrices
			\[
				A := \begin{pmatrix}
					2a_{00} & a_{01} & 0 & 0 & 0\\
					a_{01} & 2a_{11} & 0 & 0 & 0\\
					0 & 0 & 2a_{22} & a_{23} & a_{24}\\ 
					0 & 0 & a_{23} & 2a_{33} & a_{34}\\ 
					0 & 0 & a_{24} & a_{34} & 2a_{44}\\ 
				\end{pmatrix},
				\quad
				B := \begin{pmatrix}
					2b_{00} & b_{01} & 0 & 0 & 0\\
					b_{01} & 2b_{11} & 0 & 0 & 0\\
					0 & 0 & 2b_{22} & b_{23} & b_{24}\\ 
					0 & 0 & b_{23} & 2b_{33} & b_{34}\\ 
					0 & 0 & b_{24} & b_{34} & 2b_{44}\\ 
				\end{pmatrix}.
			\]
			After possibly replacing $A$ and $B$ with some linear combinations (which does not affect the corresponding del Pezzo surface), and after possibly making a linear change of coordinates on $x_0$ and $x_1$, we may assume that 
			\[
			a_{00} = a_{11} = b_{01} = 0.
			\]
			If $a_{01}b_{00}b_{11} = 0$, then the corresponding del Pezzo surface is singular. Hence, we assume that $a_{01} = b_{00} = 1$ and that $b_{11} \neq 0$. By Theorem~\ref{thm:MainExplicit}, to ensure the existence of a nonconstant Brauer class, it remains to impose that $\Norm_{\kappa(T)/k}(\eps_T) \in k^{\times 2}$. (The remaining conditions, that $\eps_T$ and $\eps_{T'}$ are nonsquares in the fields $\kappa(T)$ and $\kappa(T')$, respectively, are satisfied generically.) To this end, we introduce a new variable $w$, we let $\widetilde{A}$ and $\widetilde{B}$ be the bottom right $4\times 4$ submatrices of $A$ and $B$, respectively, and set 
			\[
				\Norm_{\kappa(T)/k}(\eps_T) = \det\left(2\sqrt{-b_{11}}\widetilde{A} + \widetilde{B}\right)\det\left(-2\sqrt{-b_{11}}\widetilde{A} + \widetilde{B}\right) = w^2,
			\]
			which we consider as a double cover of $\A^{13}$. This is the desired parameter space.
			\begin{remark}
				It is possible to obtain smaller dimensional parameter spaces whose general points correspond to degree $4$ del Pezzo surfaces together with one or three distinct nonconstant Brauer classes by further specializing the degree of the points in the degeneracy locus $\scrS$. These spaces, of course, will not parametrize all del Pezzo surfaces of degree $4$ with a nonconstant Brauer class.
			\end{remark}
			
			\subsubsection{The Brauer group of the generic fiber}
			The methods of this paper show that the generic points of the above parameter spaces have nontrivial Brauer class if the associated rank $4$ quadrics $V(Q_T)$ have $\kappa(T)$-rational points for all $T\in \scrT$.  There is evidence to suggest that this assumption is necessary absent an assumption on the cohomological dimension of the function field of the parameter space. More precisely, {our methods show that for any del Pezzo surface $X$ over any field $k$ of characteristic not $2$ such that there exists a $\scrT\subset \scrS$ satisfying~\eqref{eq:Assumption} and a $T\in\scrS\setminus\scrT$ with $\eps_T\notin \kappa(T)^{\times2}$ (e.g., the generic fiber of such parameter spaces), there is always a nontrivial element $\phi\in\HH^1(G_K, \Pic \calX_\Kbar)[2].$  }
			However, $\phi$ arises from a nontrivial class in $\Br_1 X/\Br_0 X$ if and only if it is in the kernel of the boundary map $d_2^{1,1}\colon\HH^1(G_k, \Pic X_\kbar) \to \HH^3(G_k, \kbar^\times)$ of the Hochschild-Serre spectral sequence.	
					
			The map $d_2^{1,1}$ breaks up as the composition of two maps. To see this, we fix some notation.  Let $\calD\subset \Div X_{\kbar}$ be a finite Galois invariant set of divisors that generate $\Pic X$.  Let $R = \ker\left(\calD \to \Pic X_\kbar\right)$, so that the sequences
			\begin{equation*}
				0 \to R \to \calD \to \Pic X_\kbar \to 0, \quad\textup{and}\quad
				0 \to \kbar^{\times}\to \divv^{-1}(R) \to R \to 0.
			\end{equation*}
			are short exact. Consider the associated long exact sequences in Galois cohomology. Kresch and Tschinkel~\cite[Proposition~6.1]{KreschTschinkel} show that $d_2^{1,1}$ coincides with composition of the boundary maps 
			\[
				\delta\colon\HH^1(G_k, \Pic X_\kbar) \to \HH^2(G_k, R)\quad\textup{and}\quad \partial\colon\HH^2(G_k, R) \to \HH^3(G_k, \kbar^\times).
			\] 
			
			We now specialize to the case where $\scrS$ contains two $k$-rational points $T_0$ and $T_1$, $\scrT := \{T_0, T_1\}$ satisfies~\eqref{eq:Assumption}, and there exists a $T\in\scrS\setminus\scrT$ such that $\eps_T\not\in\kappa(T)^{\times2}.$  Then there is a nontrivial element of $\phi\in\HH^1(G_k, \Pic X_\kbar)[2]$ corresponding to $\scrT$.  One can show that $\delta(\phi)$ is trivial in $\HH^2(G_k, R)$ if and only if both $W_{T_0}$ and $W_{T_1}$ have $k$-points; this computation does not depend on the field $k$. {Thus, there is hope of finding a field $k$ and a del Pezzo surface $X$ such that $\partial(\delta(\phi))$ is nontrivial in $\HH^3(G_k,\kbar^\times).$  In fact, this does occur for cubic surfaces, as shown by recent work of~\cite{Uematsu}; the above argument borrows ideas from his work.}

	\section{Arithmetic applications}\label{sec:Arithmetic}

		In this section, we restrict to the case where $k$ is a global field of characteristic different from $2$.  Otherwise, the notation remains as in~\S\ref{subsec:Notation}. We prove the results stated in the introduction (\S\S\ref{subsec:proof1}--\ref{subsec:proof3}), and explain how the results from \S\ref{sec:VerticalElements} can be used to simplify the computation of the $X(\A_k)^{\Br}$ (\S\ref{subsec:Eval}). Finally, in \S\ref{subsec:ConditionD}, we give a conditional proof that for certain del Pezzo surfaces of degree $4$ over number fields, the set of $k$-points is dense in the Brauer set.

		\begin{lemma}\label{lem:Rank4Solvable}
			Assume that $X(\A_k)\neq \emptyset$.  Then $V(Q_T)$ has a smooth $\kappa(T)$-point for all $T\in \scrS$.
		\end{lemma}
		\begin{proof}
			Let $v$ be a place of $\kappa(T)$ and let $H\subseteq\PP^4$ be a hyperplane that does not contain the vertex of $V(Q_T)$.  By assumption, there exists a point $P_v \in X(\kappa(T)_v)$, and thus $P_v\in V(Q_T)$.  Since $X$ is smooth, $P_v$ is not the vertex of $V(Q_T)$.  Consider the line $L$ that passes through $P_v$ and the vertex of $V(Q_T)$; it is defined over $\kappa(T)_v$, and so $L$ and $H$ intersect in a $\kappa(T)_v$-point $P_v'$.   Clearly, $L$ is contained in $V(Q_T)$ and so $P_v' = L \cap H \subseteq V(Q_T)\cap H$.  Thus $(V(Q_T)\cap H)(\A_{\kappa(T)})\neq \emptyset.$  Since smooth quadrics in at least $3$ variables satisfy the Hasse principle, this completes the proof.
		\end{proof}

		Therefore, if $X(\A_k)\neq\emptyset$, then we may apply the results of~\S\ref{sec:VerticalElements}.

		\subsection{Proof of Theorem~\ref{thm:MainIntro}}\label{subsec:proof1}

			Most of the theorem follows immediately from Theorem~\ref{thm:MainExplicit} and Corollary~\ref{cor:Vertical}. It remains to show that if $\#(\Br X/\Br_0 X) = 4$, then there is a map $f\colon X \dasharrow \PP^2$ such that $\Br X = \Br_{\verti}^{(f)} X$. By Theorem~\ref{thm:MainExplicit}, there exist three $k$-points $T_0$, $T_1$ and $T_2 \in \scrS$ such that every pair satisfies~\eqref{eq:Assumption}. Let $\ell_i$ be a $k$-linear form such that the associated hyperplane is tangent to $V(Q_{T_i})$ at a smooth point. Then by Theorem~\ref{thm:MainExplicit} and the definition of $\calA_{\scrT}$, the map 
			\[
				f\colon X \dasharrow \PP^2, \qquad x \mapsto [\ell_0(x):\ell_1(x):\ell_2(x)]
			\]
			has the desired property.\qed

		\subsection{Proof of Corollary~\ref{cor:FiberManifestation}}
		\label{subsec:proof2}

			Assume that there exists $t \in \PP^n(k)$ such that $f^{-1}(t)(\A_k) \neq \emptyset$; this will imply that $X(\A_k)^{\Br} \neq \emptyset$.  Fix a point $(P_v) \in f^{-1}(t)(\A_k)$, and let $U\subseteq \PP^n$ be the largest Zariski open set such that $\Br X \subseteq f^*(\Br U)$. Then for any $\calA \in \Br X$ there exists a $\calB \in \Br U$ such that $f^*\calB = \calA$. 

			Suppose that $t \in U$. Then, by functoriality of the Brauer group and class field theory, we have
			\[
				\sum_v \inv_v(\calA(P_v)) = \sum_v \inv_v((f^*\calB)(P_v))
				= \sum_v \inv_v(\calB(f(P_v))) = \sum_v \inv_v(\calB(t)) = 0
			\]
			and thus $(P_v) \in X(\A_k)^{\Br}$; see~\cite[\S5.2]{Skorobogatov-Torsors} for the definition of $X(\A_k)^{\Br}$.

			Now suppose that $t \notin U$, i.e., that $t$ is in the ramification locus of some $\calB \in \Br \kk(\PP^n)\setminus \Br k$ such that $f^*\calB \in \Br X$. By Theorem~\ref{thm:MainExplicit}, we may assume that $\calB = \calA_\scrT$ for some $\scrT \subseteq \scrS$ satisfying~\eqref{eq:Assumption}. Since $t\in\PP^n(k)$, by construction of $\calA_\scrT$,  we must have that $\#\scrT(k) = 2$.  Additionally, the inverse image of an irreducible component of the ramification locus of $\calA_\scrT$ (considered as an element of $\Br\kk(\PP^n)$) is the union of two irreducible curves, each defined over a nontrivial quadratic extension of $k$, that meet in exactly two points.  By assumption, since $f^{-1}(t)(\A_k) \neq \emptyset$, the union of these two curves contains adelic points, and thus the curves meet in two $k$-points. Hence $X(\A_k)^{\Br} \supseteq X(k) \neq \emptyset$. \qed

		\subsection{Proof of Corollary~\ref{cor:FiberEquivalence}}
		\label{subsec:proof3}

			Let $\scrT \subset \scrS$ be a degree $2$ subscheme satisfying~\eqref{eq:Assumption} such that $\calA_\scrT$ generates $\Br X/\Br_0 X$. If $\scrT(k) = \{t_0,t_1\}$ then the corollary is an application of~\cite[Thm.~2.2.1(a)]{CTSkoSD-Crelle}, taking into account that the map $f$ has at most two geometrically reducible fibers, each defined over $k$. If $\scrT$ is a single point of degree $2$ then the Corollary follows from~\cite[proof of Thm.~A]{CTSko-FibrationRevisited}.

		\subsection{Computation of the evaluation map}\label{subsec:Eval}

			If $X$ is a degree $4$ del Pezzo surface over a number field, then for any $\calA \in \Br X$ and any place $v$ of good reduction, the evaluation map $ev_{\calA}\colon X(k_v) \to \Q/\Z$ is constant~\cite{CTS-TAMS,Bright07}. The results of \S\ref{sec:VerticalElements} allow us to conclude the same is true for a larger set of places, over global fields of characteristic not 2.

			\begin{proposition}
				Let $v\nmid 2$ be a place of $k$, and let $\scrT\subseteq\scrS$ satisfy~\eqref{eq:Assumption}. If, for all $T \in \scrT$, $Q_T$ modulo $v$ is a rank $4$ quadric, then for any $\calB \in \Br X$ such that $\calB = \calA_\scrT$ in $\Br X/\Br_0 X$, the evaluation map $ev_{\calB}\colon X(k_v) \to \Q/\Z$ is constant.
			\end{proposition}	
			\begin{proof}
				Write $\F_v$ for the residue field of $k_v$. Since $Q_T$ modulo $v$ is a rank $4$ quadric, by Lemma~\ref{lem:Rank4NormalForm}, there exist $k$-linear forms $f_{T,1}, f_{T,2}, f_{T,3}, f_{T,4}$ that are linearly independent over $\F_v$ and $c\in k^{\times}$ such that
				\[
					cQ_T = f_{T,1}f_{T,2} - (f_{T,3}^2 - \eps_{T_0}f_{T,4}^2).
				\]
				Therefore, for $i_T \in \{1,2\}$, the algebras $\left( \eps_{T_0}, \ell^{-2}\prod_{T \in \scrT}\Norm_{\kappa(T)/k}(f_{T,i_T}) \right)$ are isomorphic; we write $\calA$ for this isomorphism class. Note that $\calA$ is equal to $\calA_\scrT$ in $\Br X/\Br_0 X$. 
				
				If $\eps_{T_0} \in k_v^{\times2}$, then $\calA$ is trivial when considered as an element of $\Br X_v$, and so the evaluation map $\ev_\calA$ is identically zero.  Assume that $\eps_{T_0} \notin k_v^{\times2}$; since $v \nmid 2$ and $Q_{T}$ does not drop rank modulo $v$, the extension $k(\sqrt{\eps_{T_0}})/k$ is unramified at $v$. By properness of $X$, we know that $X(k_v) = X(\calO_v)$. Then for any $P\in X(\calO_v) \subseteq V(Q_T)^{\textup{smooth}}(\calO_v)$, at least one of the $f_{T,i}(P)$ is in $\calO_v^{\times}$.  If $f_{T,3}(P)$ or $f_{T,4}(P)$ is a $v$-adic unit, then, since $\eps_{T_0}\not\in k_v^{\times2}$, both $f_{T,1}(P)$ and $f_{T,2}(P)$ are $v$-adic units.  A Hilbert symbol calculation then shows that $ev_\calA(P) = 0$.
				
				For any $\calB \in \Br X$ such that $\calB = \calA_\scrT$ in $\Br X/\Br_0 X$, we have $\calB - \calA \in \Br_0 X$; since $\ev_\calA$ is identically zero, it follows that $\ev_\calB$ is constant.
			\end{proof}

		\subsection{Density of $k$-points in the Brauer set}
		\label{subsec:ConditionD}

			In this section, we prove that certain del Pezzo surfaces of degree $4$ over a number field $k$ satisfy $\overline{X(k)} = X(\Adeles_k)^{\Br}$, i.e., the Brauer-Manin obstruction to the Hasse principle and weak approximation is the only one on $X$. This result is conditional on Schinzel's hypothesis (see~\cite[\S4]{CTSD-HPWAForSB} for its statement over number fields) and the finiteness of Tate-Shafarevich groups of elliptic curves.  Some of the cases we consider were already known by~\cite[Th\'eor\`eme~3.36(iv)]{Wittenberg-SpringerThesis}, \cite[Proposition~3.2.1(c)]{CTSkoSD-Inventiones}, and~\cite[Theorem~5.1]{HarpazSkorobogatovWittenberg}; the latter result notably does not require the use of Schinzel's hypothesis. However, our results also include some new cases, e.g., some degree $4$ del Pezzo surfaces whose associated pencil of quadrics contains 3 $k$-rational rank $4$ quadrics that are not necessarily {simultaneously} diagonalizable.
									
			We consider del Pezzo surfaces $X$ of degree $4$ over a number field $k$ with the following properties:
			\begin{enumerate}
				\item the scheme $\scrS$ contains three $k$-rational points $T_0, T_1, T_2$;
				\item the subscheme $\scrT := \{T_0,T_1\}$ of $\scrS$ satisfies~\eqref{eq:Assumption} while any other degree $2$ subschemes of $\scrS$ containing either $T_0$ or $T_1$ do not; and
				\item the line spanned by the vertices of the rank~$4$ quadrics $W_3$ and $W_4$ intersects $W_0$ and $W_1$ each in two $k$-rational points.
			\end{enumerate}
			Since the surface~\eqref{BSD}, considered by Birch and Swinnerton-Dyer in~\cite{BSD-dP4}, satisfies these conditions, we say that such a surface is of \defi{BSD type}. In fact, we will see that any surface of BSD type has a form similar to that of~\eqref{BSD}, further justifying the nomenclature.
			
			By Proposition~\ref{prop:block diagonal}, we may assume that a del Pezzo surface of degree $4$ satisfying $(1)$ is given by the intersection of two quadrics of the following form:
			\begin{align*}
				Q_0'(x_3,x_4) & = x_2^2 - \eps_0 d_0 x_1^2,\\
				Q_1'(x_3,x_4) & = x_2^2 - \eps_1 d_1 x_0^2,
			\end{align*}
			where $d_i$ has the same square class as the discriminant of $Q_i'$.  Condition $(2)$ allows us to assume that $\eps_0 = \eps_1$, and further implies, by Theorem~\ref{thm:MainExplicit}, that $\Br X/\Br k$ is generated by $\calA_\scrT$. Lastly, condition $(3)$ allows us to assume that $Q_0'$ and $Q_1'$ are split, and thus that $d_0 = d_1 = 1$.  After a change of coordinates, we may assume that $Q_0' = cx_3x_4$ and that $Q_1' = (x_3 + x_4)(ax_3 + bx_4)$.  In summary, any degree $4$ del Pezzo surface of BSD type is the intersection of two quadrics of the following form:
			\begin{equation}
				\label{eq:BSDtype}
				\begin{split}
				cx_3x_4 & = x_2^2 - \eps x_1^2,\\
				(x_3 + x_4)(ax_3 + bx_4) & = x_2^2 - \eps x_0^2.
				\end{split}
			\end{equation}
			\begin{theorem}\label{thm:BM only one}
				Let $X$ be a del Pezzo surface of degree $4$ over a number field $k$ of BSD-type. Assume Schinzel's hypothesis and the finiteness of Tate-Shafarevich groups of elliptic curves. Then $\overline{X(k)} = X(\A_k)^{\Br}$. In other words, the Brauer-Manin obstruction to the Hasse principle and weak approximation on $X$ is the only one.
			\end{theorem}
			
			\begin{proof}
				If $X(\A_k)^{\Br} = \emptyset$, then there is nothing to prove. We will show that if $X(\A_k)^{\Br} \neq \emptyset$ then $X(k) \neq \emptyset$. Work of Salberger and Skorobogatov~\cite[Theorem~0.1]{SalbergerSkorobogatov} then implies that $\overline{X(k)} = X(\A_k)^{\Br}$.
				In~\cite[Th\'eor\`eme~1.1]{Wittenberg-SpringerThesis}, Wittenberg shows that the Brauer-Manin obstruction to the Hasse principle is the only one for genus $1$ fibrations $Y \to \PP^1$ such that
				\begin{enumerate}
					\item[(i)] the generic fiber $Y_\eta$ has period $2$, i.e., it has no $\kk(\PP^1)$-rational points, but acquires a rational point over a quadratic extension of $\kk(\PP^1)$,
					\item[(ii)] the Jacobian $\Jac Y_\eta$ has $\kk(\PP^1)$-rational $2$-torsion, and
					\item[(iii)] the so-called ``condition (D)'' is satisfied (we explain this condition below).
				\end{enumerate}
				Therefore, to prove the theorem, it suffices to construct a genus $1$ fibration on (a birational model of) X, and prove that it has the desired properties.  Throughout, we assume that $X$ is of the form~\eqref{eq:BSDtype}.
				
				Let $\tilde X := \textup{Bl}_{V(x_3,x_4)} X$, and write $E_1,E_2,E_3$, and $E_4$, respectively, for the exceptional curves of the blow-up lying over the points 
				\[
					[1:1:\sqrt\eps:0:0],\; [-1:1:\sqrt\eps:0:0],\; 
					[1:-1:\sqrt\eps:0:0],\; \textup{ and } 
					[-1:-1:\sqrt\eps:0:0]
				\]
				of $V(x_3,x_4)\cap X.$ Consider the fibration $\tilde f\colon \tilde X \to \PP^1$ that sends ${\bf x}\mapsto [x_3:x_4]$.  Since the generic fiber of $\tilde f$ is the intersection of two quadrics in $\PP^3$, this is a genus $1$ fibration.  We claim that $\tilde f$ has the desired properties.
								
				By Proposition~\ref{prop:PicGens} and standard properties of blow-ups, $\Pic \tilde X$ is freely generated by
				\[
				C_0,\;C_1,\;C_2,\;C_3,\;C_4,\; B := \frac12\left(H + C_0 + C_1 + C_2 + C_3 + C_4\right),\;E_1,\;E_2,\; E_3,  \textup{ and }  E_4,
				\]
				 To determine the Picard group of $\tilde X_\eta$, we must determine the reducible fibers of $\tilde f$ and their classes in $\Pic \tilde X$.
				
				By~\cite[\S1]{CTSSDI}, the intersection of $X$ with a hyperplane $H$ is reducible if and only if $H$ is tangent to a rank $4$ quadric at a smooth point.  Using this fact, we see that the reducible fibers of $\tilde f$ lie above
				\[
					V(s t(s + t)(as + bt)(as^2 + (a + b - c)st + bt^2))\subset \PP^1_{(s:t)}.
				\]
				Using the discussion in~\S\ref{sec:PicGens} and standard properties of the blow-up, the class of a general fiber of $\tilde f$ is $F := H - E_1 - E_2 - E_3 - E_4$, and each reducible fiber contains a component whose class is among the following
				\[
					C_0 - E_1 - E_2,\; C_0 - E_3 - E_4, \;
					C_1 - E_1 - E_3,\; C_1 - E_2 - E_4, \;
					C_2 - E_1 - E_4,\; C_2 - E_2 - E_3.
				\]
				Together, these 7 classes generate the kernel of the restriction map $\Pic \tilde X \to \Pic \tilde X_\eta$. Thus, for all $i\ne j$, $2E_i - 2E_j$ is trivial in $\Pic \tilde X_\eta$ and $E_i - E_j$ is non-trivial in $\Pic \tilde X_\eta$.  Additionally, $E_1 + E_2 - E_3 - E_4$ is trivial in $\Pic \tilde X_\eta$, and therefore $(\Jac \tilde X_\eta )[2] = \langle E_1 - E_2, E_1 - E_3 \rangle.$  Since each element of the absolute Galois group $G_k$ either fixes all exceptional curves, or simulatenously interchanges $E_1$ with $E_4$ and $E_2$ with $E_3$, we deduce that every $2$-torsion class in $\Jac \tilde X_\eta$ is $\kk(\PP^1)$-rational. In other words, $\tilde f$ satisfies condition (ii).
				
				Next, we turn to condition (i). The class $E_1 + E_4$ in $\Pic \tilde X$ is Galois invariant and gives rise to a degree $2$ point of the generic fiber $\tilde X_\eta$; therefore the period of $\tilde X_\eta$ divides $2$.  It remains to show that the period of $\tilde X_\eta$ is not $1$.  Consider the following intersection numbers:
				\[
					F\cdot E_i = 1, \quad
					F\cdot C_j = 2, \quad \textup{ and }
					F\cdot B = 7.
				\]
				If $\tilde X_\eta$ has a $\kk(\PP^1)$-rational point, then there is a Galois invariant degree $1$ element of $\Pic\tilde X_\eta.$  Since the restriction map $\Pic \tilde X \to \Pic \tilde X_\eta$ is surjective, we may compute $\Pic\tilde X_\eta$, as a Galois module, using the generators and relations given above. In particular, we deduce that $\Pic \tilde X_\eta$ is generated by 
				\[
					B,\; 
					C_3, \; E_3, \;
					(E_1 - E_3),\;\textup{ and }(E_2 - E_3),
				\]
				with relations $2(E_2 - E_3) = 2(E_1 - E_3) = 0$.  By BSD type condition (2), there exists an element $\sigma\in G_k$ such that $\sigma(C_0) = C_0'$, $\sigma(C_1) = C_1'$ and $\sigma(C_2) = C_2$; this also implies that $\sigma(E_3) = E_2$.  One can check that, regardless of the action of $\sigma$ on $C_3$ and $C_4$, we have
				\[
					\sigma(B) = \beta B + \gamma C_3 + \delta E_3
					+ (E_2 - E_3),\; 
					{\sigma(C_3) = \beta' B + \gamma' C_3 + \delta' E_3, }
					 \;\textup{ and }\;
					\sigma(E_3) = E_3 + (E_2 - E_3)
				\]
				 for some integers {$\beta,\beta',\gamma,\gamma',\delta$ and $\delta'$.}  Therefore, if some integer combination of $B, C_3, E_3, (E_2 - E_3)$ and $(E_1 - E_3)$ is fixed by $G_k$, then the integer coefficients of $B$ and $E_3$ must have the same parity.  However, such a divisor class has even intersection with $F$, and so there is no Galois invariant divisor of degree $1$.  In particular, this implies that $\tilde X_\eta$ has period {equal to $2$} and condition (i) is satisfied.
				
				Finally, we show that $\tilde f$ satisfies condition (D). {Let $[\alpha:1]\in \PP^1_{(s:t)}$ be such that $\tilde f^{-1}([\alpha:1])$ is smooth and set $U := \PP^1\setminus[\alpha:1]$.} Let $\Br^{\textup{gnr}/\tilde X}\tilde X_{U}$ denote the group of classes in $\Br \tilde X_{U}$ generated by classes \emph{geometrically} unramified over $\tilde X$. By~\cite[Prop.\ 1.48]{Wittenberg-SpringerThesis}, condition (D) on the fibration $\tilde f$ is equivalent to 
		\begin{equation*}
		\label{eq: D1}
			\left(\Br^{\textup{gnr}/\tilde X}\tilde X_{U}\right)\{2\} \subset \Br_{\verti}^{(\tilde f)} \tilde X_{U}.
		\end{equation*}
		Since $\tilde X$ is geometrically rational, we have $\Br^{\textup{gnr}/\tilde X} \tilde X_{U} = \Br_1 \tilde X_{U}$. It is thus enough to show that
		\[
			\frac{\Br_1 \tilde X_{U}}{\Br k}\{2\} \subset \frac{\Br_{\verti}^{(\tilde f)}\tilde X_{U}}{\Br k },
		\]
		By the sequence of low-degree terms of the Hochschild-Serre spectral sequence, the group $\Br_1 \tilde X_U/\Br k$ is isomorphic to ${\rm H}^1\left(G_k,\Pic {\tilde X_\Ubar}\right)$. Note that $\Pic {\tilde X_\Ubar} \cong \Pic \tilde X / \langle F\rangle$ as Galois modules, so we can use our explicit generators for $\Pic \tilde X$ to compute this cohomology group, once we know the Galois actions on these generators.  Since $X$ is of BSD type, the sets 
		\[
			\{E_1,E_4\},\; \{E_2,E_3\},\;\{C_0, C_0'\},\; \{C_1,C_1'\},\;
			\{C_2,C_2'\},\;\textup{ and }\;\{C_3,C_3', C_4,C_4'\}
		\]
		are all invariant under $G_k$, and by~\eqref{eq:product}, the definition of $E_i$ and~\eqref{eq:Assumption}, the action of $G_k$ on $\{C_0, C_0'\}, $ and $\{C_3,C_3',C_4,C_4'\}$ determines the action on $\{C_1,C_1'\},\{C_2, C_2'\}, \{E_1,E_4\},$ and $\{E_2,E_3\}$.  Therefore, the action of $G_k$ factors through a subgroup of $\Z/2\times D_4$, and, in each case, one can compute, either using \texttt{Magma}~\cite{Magma} or by hand using the last isomorphism in~\eqref{eq:Maps}, that 
		\[
			{\rm H}^1\left(G_k, \Pic {\tilde X_\Ubar}\right)
			\isom {\rm H}^1\left(G_k,\Pic X_\kbar\right)\times \Z/2.
		\]
		On the other hand, ${(\eps, 1 - \alpha x_4/x_3)}\in \Br \tilde X_U$ and, since it is ramified on $F$, it is nontrivial in $\Br \kk(\tilde X)$ and not contained in $\Br \tilde X$.  Therefore, $\frac{\Br_1 \tilde X_{U}}{\Br k}$ is generated by $\calA_\scrT = (\eps, x_4/x_3 + 1)$ and ${(\eps, 1 - \alpha x_4/x_3)}$, both of which are plainly vertical.  Hence condition (D) holds for $\tilde f$, as desired, which completes the proof.
			\end{proof}
	
	\section{Brauer groups of order $4$}\label{sec:Order4}

		It is natural to ask whether Theorem~\ref{thm:MainIntro} can be strengthened to show that $\Br X = \Br_{\verti}^{(f)} X$ for some map $f\colon X \dasharrow \PP^1$, even if $\#(\Br X/\Br_0 X) = 4$. In this section we consider a variation of this question for maps obtained by projecting away from a plane in $\PP^4$. 

		We retain the notation from \S\ref{sec:Picdp4}; throughout, we assume that $\#(\Br X/\Br_0 X) = 4$. By Theorem~\ref{thm:MainExplicit}, there exist three $k$-points $T_0$, $T_1$, and $T_2 \in \scrS(k)$ such that each pair satisfies~\eqref{eq:Assumption}. In this case, the three nontrivial elements of $\Br X/\Br_0 X$ are represented by the quaternion algebras 
		\[
			\{ (\eps_{T_0},\ell_i/\ell_j) : 0 \leq i < j \leq 2\},
		\]
		where $\ell_i$ is any $k$-linear form whose corresponding hyperplane is tangent to $V(Q_{T_i})$ at a smooth point. 

		We ask if there are forms $\ell_0, \ell_1$ and $\ell_2$ as above and a rational map $f\colon X \dasharrow \PP^1$, obtained by projecting away from a plane, such that $\ell_i/\ell_j \in f^*(\kk(\PP^1))$. If so, then $\Br X = \Br_{\verti}^{(f)} X$ for this map. If $\ell_i/\ell_j \in f^*(\kk(\PP^1))$, then, possibly after changing coordinates, the map $f$ is given by $x \mapsto [\ell_i(x):\ell_j(x)]$. Thus, a positive answer to the question is equivalent to $k$-linear dependence of $\ell_0, \ell_1$, and $\ell_2$. 

		\begin{prop}\label{prop:Vertical4}
			Let $X$ be a del Pezzo surface of degree $4$ over $k$. Assume that $X(k) \neq \emptyset$  and that $\#(\Br X/\Br_0 X) = 4$.  Then there exists a map $f \colon X \dasharrow \PP^1$, obtained by projection away from a plane, such that $\Br X = \Br_{\verti}^{(f)} X$.
		\end{prop}
		\begin{proof}
			Let $P\in X(k)$ and let $\ell_0, \ell_1$ be linear forms such that the corresponding hyperplanes are tangent to $V(Q_{T_0})$ and  $V(Q_{T_1})$ at $P$, respectively.  Define the map
			\[
				f\colon X\dasharrow \PP^1\qquad x\mapsto [\ell_0(x):\ell_1(x)].
			\]
			Since $Q_{T_2}$ is a $k$-linear combination of $Q_{T_0}$ and $Q_{T_1}$, there is a $k$-linear combination $\ell_2$ of $\ell_0$ and $\ell_1$ whose corresponding hyperplane is tangent to $V(Q_{T_2})$ at $P$. The definition of $\calA_\scrT$ and Theorem~\ref{thm:MainExplicit} show that $\Br X = \Br_{\verti}^{(f)} X$.
		\end{proof}

		\begin{remark}
			The proof of Proposition~\ref{prop:Vertical4} gives a map $f\colon X \dasharrow \PP^1$ whose generic fiber is singular. To see this, let $\ell$ be a $k$-linear combination of $\ell_0$ and $\ell_1$.  Then there is some quadric $Q$ in the pencil of quadrics defining $X$ such that $V(\ell)$ is tangent to $V(Q)$ at a smooth point $P$.  The Jacobian criterion then shows that $X\cap V(\ell)$ is singular at $P$.
		\end{remark}

		One wonders if Proposition~\ref{prop:Vertical4} is as strong as possible: does the conclusion hold without assuming that $X(k) \neq \emptyset$? Is it possible to construct a map $f$ whose generic fiber is not singular? After rephrasing these questions more precisely, we will see that they are essentially the same question.

		The linear dependence condition on  $\ell_0$, $\ell_1$, and $\ell_2$ gives rise to a determinantal subvariety $Y^{(X)}$ in a product of three quadric surfaces, as follows. Say $Q_{T_i} \in k[x_0,\dots,x_4]$, and let $P_i \in \PP^4(\kbar)$ for $i = 0,1,2$. Consider the matrix $M(P_0,P_1,P_2)$ whose $(i,j)$-th entry is
		\[
		\frac{\partial Q_{T_i}}{\partial x_j}(P_i).
		\]
		Let $H\subseteq \PP^4$ be a hyperplane avoiding the vertices of $V(Q_{T_i})$ for $0\leq i \leq 2$. Then define
		\[
		Y^{(X)} := \{ (P_0,P_1,P_2) \in H^3 : \rk M(P_0,P_1,P_2) \leq 2 \textup{ and } Q_{T_i}(P_i) = 0 \}.
		\]
		The existence of $k$-linearly dependent forms $\ell_0$, $\ell_1$, and $\ell_2$ is equivalent to $Y^{(X)}(k) \neq \emptyset$: if $(P_0,P_1,P_2) \in Y^{(X)}(k)$, then we take $\ell_i$ to be the linear form defining the hyperplane tangent to $V(Q_{T_i})$ at $P_i$, for $0 \leq i \leq 2$.

		There is an embedding of $X(k)$ in $Y^{(X)}(k)$: for a point $x \in X(k)$, let $P_i(x) \in H(k)$ be the intersection of $H$ with the line joining $x$ and the vertex of $V(Q_{T_i})$. Then $(P_0(x),P_1(x),P_2(x))$ is a $k$-point of $Y^{(X)}$; this is the point giving rise to the map $f$ described in the proof of Proposition~\ref{prop:Vertical4}.  From this discussion we see that the questions above are equivalent to the following:

		\begin{question}
			Is there a $k$-rational point on $Y^{(X)}$ that does not come from the embedding of $X(k)$ as described above?
		\end{question}

		\begin{remark}\label{rmk:VerticalConic}
			If we relax the condition that $f$ is obtained by projecting away from a plane, then the existence of a $k$-point implies that $\Br X = \Br_{\verti}^{(f)} X$ where $f$ is a generically smooth map to $\PP^1.$  The idea is as follows.

			Let $P\in X(k)$ and consider $\widetilde{X}$, the blow-up of $X$ at $P$.  The surface $\widetilde{X}$ is a del Pezzo surface of degree $3$ containing a line $L$ defined over $k$.  Embed $\widetilde{X}$ into $\PP^3$ by the anticanonical embedding and consider the pencil of planes that contain $L$.  This gives a rational map $f\colon X\dasharrow \PP^1$ whose generic fiber is a smooth conic.  Since the generic fiber is a smooth conic, we have $\Br X = \Br_{\verti}^{(f)}(X).$
		\end{remark}



\begin{bibdiv}
		\begin{biblist}

			\bib{BenderSD}{article}{
			   author={Bender, A. O.},
			   author={Swinnerton-Dyer, Peter},
			   title={Solubility of certain pencils of curves of genus 1, and of the
			   intersection of two quadrics in $\Bbb P^4$},
			   journal={Proc. London Math. Soc. (3)},
			   volume={83},
			   date={2001},
			   number={2},
			   pages={299--329},
			   issn={0024-6115},
			   review={\MR{1839456 (2002e:14033)}},
			   doi={10.1112/S0024611501012795},
			}

			\bib{BSD-dP4}{article}{
			   author={Birch, B. J.},
			   author={Swinnerton-Dyer, H. P. F.},
			   title={The Hasse problem for rational surfaces},
			   note={Collection of articles dedicated to Helmut Hasse on his
			   seventy-fifth birthday, III},
			   journal={J. Reine Angew. Math.},
			   volume={274/275},
			   date={1975},
			   pages={164--174},
			   issn={0075-4102},
			   review={\MR{0429913 (55 \#2922)}},
			}

			\bib{Magma}{article}{
			   author={Bosma, Wieb},
			   author={Cannon, John},
			   author={Playoust, Catherine},
			   title={The Magma algebra system. I. The user language},
			   note={Computational algebra and number theory (London, 1993)},
			   journal={J. Symbolic Comput.},
			   volume={24},
			   date={1997},
			   number={3-4},
			   pages={235--265},
			   issn={0747-7171},
			   review={\MR{1484478}},
			   doi={10.1006/jsco.1996.0125},
			}

			\bib{Bright07}{article}{
			   author={Bright, Martin},
			   title={Efficient evaluation of the Brauer-Manin obstruction},
			   journal={Math. Proc. Cambridge Philos. Soc.},
			   volume={142},
			   date={2007},
			   number={1},
			   pages={13--23},
			   issn={0305-0041},
			   review={\MR{2296387 (2007k:14026)}},
			   doi={10.1017/S0305004106009844},
			}

			\bib{BBFL}{article}{
			   author={Bright, M. J.},
			   author={Bruin, N.},
			   author={Flynn, E. V.},
			   author={Logan, A.},
			   title={The Brauer-Manin obstruction and Sh[2]},
			   journal={LMS J. Comput. Math.},
			   volume={10},
			   date={2007},
			   pages={354--377 (electronic)},
			   issn={1461-1570},
			   review={\MR{2342713 (2008i:11087)}},
			}
			
			\bib{CT-BenderSD}{article}{
			   author={Colliot-Th{\'e}l{\`e}ne, Jean-Louis},
			   title={Hasse principle for pencils of curves of genus one whose Jacobians
			   have a rational 2-division point, close variation on a paper of Bender
			   and Swinnerton-Dyer},
			   language={English, with French summary},
			   conference={
		      title={Rational points on algebraic varieties},
			   },
			   book={
			      series={Progr. Math.},
			      volume={199},
			      publisher={Birkh\"auser},
			      place={Basel},
			   },
			   date={2001},
			   pages={117--161},
			   review={\MR{1875172 (2003f:14017)}},
			}


			\bib{CTHS-NormicBundle}{article}{
			   author={Colliot-Th{\'e}l{\`e}ne,  Jean-Louis},
			   author={Harari, David},
			   author={Skorobogatov, Alexei},
			   title={Valeurs d'un polyn\^ome \`a une variable repr\'esent\'ees
			 			par une norme},
			   language={French, with English summary},
			   conference={
			      title={Number theory and algebraic geometry},
			   },
			   book={
			      series={London Math. Soc. Lecture Note Ser.},
			      volume={303},
			      publisher={Cambridge Univ. Press},
			      place={Cambridge},
			   },
			   date={2003},
			   pages={69--89},
			   review={\MR{2053456 (2005d:11095)}},
			}
			
			\bib{CTS80}{article}{
			   author={Colliot-Th{\'e}l{\`e}ne, Jean-Louis},
			   author={Sansuc, Jean-Jacques},
			   title={La descente sur les vari\'et\'es rationnelles},
			   language={French},
			   conference={
			      title={Journ\'ees de G\'eom\'etrie Alg\'ebrique d'Angers, Juillet
			      1979/Algebraic Geometry, Angers, 1979},
			   },
			   book={
			      publisher={Sijthoff \& Noordhoff},
			      place={Alphen aan den Rijn},
		   },
			   date={1980},
			   pages={223--237},
			   review={\MR{605344 (82d:14016)}},
			}

			\bib{CTSansuc-Schinzel}{article}{
			   author={Colliot-Th{\'e}l{\`e}ne, Jean-Louis},
			   author={Sansuc, Jean-Jacques},
			   title={Sur le principe de Hasse et l'approximation faible, et sur une
			   hypoth\`ese de Schinzel},
			   language={French},
			   journal={Acta Arith.},
			   volume={41},
			   date={1982},
			   number={1},
			   pages={33--53},
			   issn={0065-1036},
			   review={\MR{667708 (83j:10019)}},
			}			
			
			\bib{CTSSDI}{article}{
			   author={Colliot-Th{\'e}l{\`e}ne, Jean-Louis},
			   author={Sansuc, Jean-Jacques},
			   author={Swinnerton-Dyer, Peter},
			   title={Intersections of two quadrics and Ch\^atelet surfaces. I},
			   journal={J. Reine Angew. Math.},
			   volume={373},
			   date={1987},
			   pages={37--107},
			   issn={0075-4102},
			   review={\MR{870307 (88m:11045a)}},
			}
				
			\bib{CTSko-FibrationRevisited}{article}{
			   author={Colliot-Th{\'e}l{\`e}ne,  Jean-Louis},
			   author={Skorobogatov, Alexei},
			   title={Descent on fibrations over ${\bf P}^1_k$ revisited},
			   language={English, with French summary},
			   journal={Math. Proc. Cambridge Philos. Soc.},
			   volume={128},
			   date={2000},
			   number={3},
			   pages={383--393},
			   issn={0305-0041},
			   review={\MR{1744112 (2000m:11052)}},
			   doi={10.1017/S0305004199004077},
			}
			
			\bib{CTS-TAMS}{article}{
			   author={Colliot-Th{\'e}l{\`e}ne,  Jean-Louis},
			   author={Skorobogatov, Alexei},
			   title={Good reduction of the Brauer-Manin obstruction},
			   journal={Trans. Amer. Math. Soc.},
			   date={2012},
			   note={posted on September 19, 2012 (to appear in print)},
			}
			
			\bib{CTSkoSD-Inventiones}{article}{
			   author={Colliot-Th{\'e}l{\`e}ne,  Jean-Louis},
			   author={Skorobogatov, Alexei},
			   author={Swinnerton-Dyer, Peter},
			   title={Hasse principle for pencils of curves of genus one whose Jacobians
			   have rational $2$-division points},
			   journal={Invent. Math.},
			   volume={134},
			   date={1998},
			   number={3},
			   pages={579--650},
			   issn={0020-9910},
			   review={\MR{1660925 (99k:11095)}},
			   doi={10.1007/s002220050274},
			}

			
			\bib{CTSkoSD-Crelle}{article}{
			   author={Colliot-Th{\'e}l{\`e}ne,  Jean-Louis},
			   author={Skorobogatov, Alexei},
			   author={Swinnerton-Dyer, Peter},
			   title={Rational points and zero-cycles on fibred varieties: Schinzel's
			   hypothesis and Salberger's device},
			   journal={J. Reine Angew. Math.},
			   volume={495},
			   date={1998},
			   pages={1--28},
			   issn={0075-4102},
			   review={\MR{1603908 (99i:14027)}},
			   doi={10.1515/crll.1998.019},
			}

			\bib{CTSD-HPWAForSB}{article}{
			   author={Colliot-Th{\'e}l{\`e}ne, Jean-Louis},
			   author={Swinnerton-Dyer, Peter},
			   title={Hasse principle and weak approximation for pencils of
			   Severi-Brauer and similar varieties},
			   journal={J. Reine Angew. Math.},
			   volume={453},
			   date={1994},
			   pages={49--112},
			   issn={0075-4102},
			   review={\MR{1285781 (95h:11060)}},
			   doi={10.1515/crll.1994.453.49},
			}
			
			\bib{Grothendieck-BrauerIII}{article}{
			   author={Grothendieck, Alexander},
			   title={Le groupe de Brauer. III. Exemples et compl\'ements},
			   language={French},
			   conference={
			      title={Dix Expos\'es sur la Cohomologie des Sch\'emas},
			   },
			   book={
			      publisher={North-Holland},
			      place={Amsterdam},
			   },
			   date={1968},
			   pages={88--188},
			   review={\MR{0244271 (39 \#5586c)}},
			}
			
			\bib{HarpazSkorobogatovWittenberg}{article}{
			   author={Harpaz, Yonatan},
			   author={Skorobogatov, Alexei},
			   author={Wittenberg, Olivier},
			   title={The Hardy-Littlewood conjecture and rational points},
			   note={Preprint available at {\tt arXiv:1304.3333}}
			}
			
			\bib{KreschTschinkel}{article}{
			   author={Kresch, Andrew},
			   author={Tschinkel, Yuri},
			   title={Effectivity of Brauer-Manin obstructions},
			   journal={Adv. Math.},
			   volume={218},
			   date={2008},
			   number={1},
			   pages={1--27},
			   issn={0001-8708},
			   review={\MR{2409407 (2009e:14038)}},
			   doi={10.1016/j.aim.2007.11.017},
			}

			\bib{KST-dP4s}{article}{
			   author={Kunyavski{\u\i}, B. {\`E}.},
			   author={Skorobogatov, A. N.},
			   author={Tsfasman, M. A.},
			   title={del Pezzo surfaces of degree four},
			   language={English, with French summary},
			   journal={M\'em. Soc. Math. France (N.S.)},
			   number={37},
			   date={1989},
			   pages={113},
			   issn={0037-9484},
			   review={\MR{1016354 (90k:14035)}},
			}
			
			\bib{Manin-ICM}{article}{
			   author={Manin, Yu. I.},
			   title={Le groupe de Brauer-Grothendieck en g\'eom\'etrie 
					diophantienne},
			   conference={
			      title={Actes du Congr\`es International des Math\'ematiciens},
			      address={Nice},
			      date={1970},
			   },
			   book={
			      publisher={Gauthier-Villars},
			      place={Paris},
			   },
			   date={1971},
			   pages={401--411},
			   review={\MR{0427322 (55 \#356)}},
			}

			\bib{Manin-CubicForms}{book}{
			   author={Manin, Yu. I.},
			   title={Cubic forms: algebra, geometry, arithmetic},
			   note={Translated from the Russian by M. Hazewinkel;
			   North-Holland Mathematical Library, Vol. 4},
			   publisher={North-Holland Publishing Co.},
			   place={Amsterdam},
			   date={1974},
			   pages={vii+292},
			   isbn={0-7204-2456-9},
			   review={\MR{0460349 (57 \#343)}},
			}
			
			\bib{SalbergerSkorobogatov}{article}{
			   author={Salberger, P.},
			   author={Skorobogatov, A. N.},
			   title={Weak approximation for surfaces defined by two quadratic forms},
			   journal={Duke Math. J.},
			   volume={63},
			   date={1991},
			   number={2},
			   pages={517--536},
			   issn={0012-7094},
			   review={\MR{1115119 (93e:11079)}}
			}

			\bib{Skorobogatov-Torsors}{book}{
			   author={Skorobogatov, Alexei},
			   title={Torsors and rational points},
			   series={Cambridge Tracts in Mathematics},
			   volume={144},
			   publisher={Cambridge University Press},
			   place={Cambridge},
			   date={2001},
			   pages={viii+187},
			   isbn={0-521-80237-7},
			   review={\MR{1845760 (2002d:14032)}},
			   doi={10.1017/CBO9780511549588},
			}
			
			\bib{SD-BrauerGroup}{article}{
			   author={Swinnerton-Dyer, Peter},
			   title={The Brauer group of cubic surfaces},
			   journal={Math. Proc. Cambridge Philos. Soc.},
			   volume={113},
			   date={1993},
			   number={3},
			   pages={449--460},
			   issn={0305-0041},
			   review={\MR{1207510 (94a:14038)}},
			   doi={10.1017/S0305004100076106},
			}
			
			\bib{SD-TwoQuadrics}{article}{
			   author={Swinnerton-Dyer, Peter},
			   title={Rational points on certain intersections of two quadrics},
			   conference={
		       title={Abelian varieties},
		      address={Egloffstein},
		      date={1993},
			   },
			   book={
		      publisher={de Gruyter},
		      place={Berlin},
			   },
			   date={1995},
			   pages={273--292},
			   review={\MR{1336612 (97a:11099)}},
			}

			\bib{SD99}{article}{
			   author={Swinnerton-Dyer, Peter},
			   title={Brauer-Manin obstructions on some Del Pezzo surfaces},
			   journal={Math. Proc. Cambridge Philos. Soc.},
			   volume={125},
			   date={1999},
			   number={2},
			   pages={193--198},
			   issn={0305-0041},
			   review={\MR{1643855 (99h:11071)}},
			   doi={10.1017/S0305004198003144},
			}			
			
			\bib{Uematsu}{article}{
			   author={Uematsu, Tetsuya},
			   title={On the Brauer group of diagonal cubic surfaces},
			   journal={Q. J. Math.},
			   date={2013},
			   doi={10.1093/qmath/hat013}
			}
			
			\bib{VA-WeakApprox}{article}{
			   author={V{\'a}rilly-Alvarado, Anthony},
			   title={Weak approximation on del Pezzo surfaces of degree 1},
			   journal={Adv. Math.},
			   volume={219},
			   date={2008},
			   number={6},
			   pages={2123--2145},
			   issn={0001-8708},
			   review={\MR{2456278 (2009j:14045)}},
			   doi={10.1016/j.aim.2008.08.007},
			}


			\bib{Wittenberg-SpringerThesis}{book}{
			   author={Wittenberg, Olivier},
			   title={Intersections de deux quadriques et pinceaux de courbes de 
				genre 1/Intersections of two quadrics and pencils of curves of 
				genus 1},
			   language={French},
			   series={Lecture Notes in Mathematics},
			   volume={1901},
			   publisher={Springer},
			   place={Berlin},
			   date={2007},
			   pages={viii+218},
			   isbn={978-3-540-69137-2},
			   isbn={3-540-69137-5},
			   review={\MR{2307807 (2008b:14029)}},
			}

		\end{biblist}
	\end{bibdiv}
\end{document}